\documentclass[reqno]{amsart}
\usepackage{amscd,amsfonts,amssymb, amsmath}
\usepackage{graphicx, color}

\numberwithin{equation}{section}
\newtheorem{Theorem}{Theorem}[section]
\newtheorem{Lemma}{Lemma}[section]

\theoremstyle{definition}
\newtheorem{Definition}{Definition}[section]
\theoremstyle{remark}
\newtheorem{Remark}{Remark}[section]
\newtheorem{Proposition}{Proposition}[section]

\renewcommand{\r}{\rho}
\renewcommand{\t}{\theta}

\renewcommand{\u}{{\bf u}}
\renewcommand{\H}{{\bf H}}
\newcommand{\R}{{\mathbb R}}
\newcommand{\Dv}{{\rm div}}

\newcommand{\m}{{\bf m}}

\newcommand{\na}{\nabla}
\newcommand{\dl}{\delta}
\newcommand{\E}{{\mathcal E}}
\def\f{\frac}
\renewcommand{\O}{\Omega}
\def\ov{\overline}

\def\hf1{^\f{1}{1-\xi^2}}

\def\be{\begin{equation}}
\def\en{\end{equation}}
\def\bs{\begin{split}}
\def\es{\end{split}}

\title[Global Solutions to Magnetohydrodynamic Flows]
{Global solutions to the three-dimensional full compressible
magnetohydrodynamic flows}

\author{Xianpeng HU and Dehua Wang}
\address{Department of Mathematics, University of Pittsburgh,
                           Pittsburgh, PA 15260, USA.}
\email{xih15@pitt.edu}
\address{Department of Mathematics, University of Pittsburgh,
                           Pittsburgh, PA 15260, USA.}
\email{dwang@math.pitt.edu}


\begin{document}

\begin{abstract}
The equations of the three-dimensional viscous, compressible, and
heat conducting magnetohydrodynamic flows are considered in a
bounded domain. The viscosity coefficients and heat conductivity
can depend on the temperature. A solution to the initial-boundary
value problem  is constructed through an approximation scheme and
a weak convergence method. The existence  of  a global variational
weak solution to the three-dimensional full magnetohydrodynamic
equations with large data is established.
\end{abstract}

\maketitle

\section{Introduction}

Magnetohydrodynamics, or MHD,  studies the dynamics of
electrically conducting fluids and the theory of the macroscopic
interaction of electrically conducting fluids with a magnetic
field. The applications of magnetohydrodynamics cover a very wide
range of physical areas from liquid metals to cosmic plasmas, for
example, the  intensely heated and ionized fluids in an
electromagnetic field in astrophysics, geophysics, high-speed
aerodynamics, and plasma physics. Astrophysical problems include
solar structure, especially in the outer layers, the solar wind
bathing the earth and other planets, and interstellar magnetic
fields. The primary geophysical problem is planetary magnetism,
produced by currents deep in the planet, a problem that has not
been solved to any degree of satisfaction. Magnetohydrodynamics is
of importance in connection with many engineering problems as
well, such as sustained plasma confinement for controlled
thermonuclear fusion, liquid-metal cooling of nuclear reactors,
magnetohydrodynamic power generation, electro-magnetic casting of
metals, and plasma accelerators for ion thrusters for spacecraft propulsion.
Due to their practical relevance, magnetohydrodynamic
 problems have long been the subject of intense
cross-disciplinary research, but except for relatively simplified
special cases, the rigorous mathematical analysis of such problems
remains open.

In magnetohydrodynamic flows, magnetic fields can induce currents
in a moving conductive fluid, which create forces on the fluid,
and also change the magnetic field itself.
 There is a complex interaction between
the magnetic and fluid dynamic phenomena, and both hydrodynamic
and electrodynamic effects have to be considered. The set of
equations which describe compressible viscous magnetohydrodynamics
are a combination of the compressible Navier-Stokes equations of
fluid dynamics and Maxwell's equations of electromagnetism. In
this paper, we consider the full system of partial differential
equations for the three-dimensional viscous compressible
magnetohydrodynamic flows in the Eulerian coordinates (\cite{KL,LL}):
\begin{subequations} \label{1}
\begin{align}
&\r_t +\Dv(\r\u)=0, \label{11} \\
&(\r\u)_t+\Dv\left(\r\u\otimes\u\right)+\nabla p
  =(\na \times \H)\times \H+\Dv\Psi, \label{12} \\
&\E_t+\Dv\big(\u(\E'+p)\big)
=\Dv\big((\u\times\H)\times\H+\nu\H\times(\nabla\times\H)+\u\Psi+\kappa\nabla\theta\big),
  \label{13} \\
&\H_t-\nabla\times(\u\times\H)=-\nabla\times(\nu\nabla\times\H),\qquad
\Dv\H=0,\label{14}
\end{align}
\end{subequations}
where  $\r$ denotes the density, $\u\in \R^3$ the velocity, 
$\H\in \R^3$ the magnetic field, and $\theta$ the temperature; $\Psi$ is the viscous stress tensor given by
$$\Psi=\mu(\nabla\u+\nabla\u^T)+\lambda\,\Dv\u\,\mathbf{I},$$
and $\E$ is the total
energy given by $$\E=\r\left(e+\f{1}{2}|\u|^2\right)+\f{1}{2}|\H|^2
\textrm{ and } \E'=\r\left(e+\f{1}{2}|\u|^2\right),$$ with $e$ the
internal energy, $\f{1}{2}\r|\u|^2$ the kinetic energy, and
$\f{1}{2}|\H|^2$ the magnetic energy. The equations of state
$p=p(\r,\theta)$, $e=e(\r,\theta)$ relate the pressure $p$ and the
internal energy $e$ to the density and the temperature of the flow;
$\mathbf{I}$ is the $3\times 3$ identity matrix, and $\nabla\u^T$ is
the transpose of the matrix $\nabla\u$. The viscosity coefficients
$\lambda, \mu$ of the flow satisfy $2\mu+3\lambda>0$ and $\mu>0$; $\nu>0$ is the magnetic
diffusivity acting as a magnetic diffusion coefficient of the
magnetic field,  $\kappa>0$ is the heat
conductivity. Equations \eqref{11}, \eqref{12},
\eqref{13} describe the conservation of mass,  momentum,
and energy, respectively. It is well-known that the
electromagnetic fields are governed by the Maxwell's equations. In
magnetohydrodynamics, the displacement current  can be neglected
(\cite{KL, LL}). As a consequence, the equation \eqref{14} is called
the induction equation, and the electric field can be written in
terms of the magnetic field $\H$ and the velocity $\u$,
\begin{equation*} 
{\bf E}=\nu\nabla\times\H - \u\times\H.
\end{equation*}
Although the electric field ${\bf E}$ does not appear in the MHD
system \eqref{1}, it is indeed induced according to the
above relation by the moving conductive flow in the magnetic field.

There have been a lot of studies on magnetohydrodynamics by
physicists and mathematicians because of its physical importance,
complexity, rich phenomena, and mathematical challenges; see
\cite{gw, gw2, f6, FJN, h1, hw1, HT, LL, w1} and the references
cited therein. In particular, the one-dimensional problem has been
studied in many papers, for examples, \cite{gw, gw2,
FJN,HT,sm,tz,w1} and so on. However, many fundamental problems for
MHD are still open. For example, even for the one-dimensional
case, the global existence of classical solutions to the full
perfect MHD equations with large data remains unsolved when all
the viscosity, heat conductivity, 
and magnetic diffusivity coefficients are
constant, although the corresponding problem for the Navier-Stokes
equations was solved in \cite{KS} long time ago. The reason is
that the presence of the magnetic field and its interaction with
the hydrodynamic motion in the MHD flow of large oscillation cause
serious difficulties. In this paper we consider the global weak
solution to the three-dimensional MHD problem with large data, and
investigate the fundamental problem of global existence.

More precisely, we study the initial-boundary value problem of
\eqref{1} in a bounded spatial domain $\O\subset\R^3$ with the initial data:
\begin{equation}\label{116}
(\r, \r\u, \H, \t)|_{t=0}=(\r_0, m_0, \H_0, \t_0)(x), \quad x\in\O,
\end{equation}
and  the no-slip boundary conditions on the velocity and the magnetic field, and
the thermally insulated boundary condition on the heat flux $q=-\kappa\nabla\theta$:
\begin{equation}\label{17}
\u|_{\partial\O}=0, \quad \H|_{\partial\O}=0, \quad q|_{\partial\O}=0.
\end{equation}
The aim of this paper is to construct the solution of the
initial-boundary value problem of \eqref{1}-\eqref{17} and
establish the global existence theory of variational weak
solutions. In Hu-Wang \cite{hw1}, we studied global weak solutions
to the initial-boundary value problem of the isentropic case for
the three-dimensional MHD flow, while in this paper we study the
full nonisentropic case. We are interested in the case that the
viscosity and heat conductivity coefficients $\mu=\mu(\theta),
\lambda=\lambda(\t), \kappa=\kappa(\t)$
 are positive functions of the temperature $\theta$; and
the magnetic diffusivity coefficient $\nu>0$ is assumed to be a constant in order to avoid
unnecessary technical details.
As for the pressure $p=p(\r,\t)$, it will be determined through a general constitutive equation:
\begin{equation} \label{18}
p=p(\r,\t)=p_e(\r)+\t p_{\t}(\r)
\end{equation}
for certain functions $p_e$, $p_{\t}\in C[0,\infty)\cap
C^1(0,\infty)$. The basic principles of classical thermodynamics
imply that the internal energy $e$ and pressure $p$ are
interrelated through Maxwell's relationship:
\begin{equation*}
\frac{\partial e}{\partial
\r}=\frac{1}{\r^2}\left(p-\t\frac{\partial p}{\partial \t}\right),\quad
\frac{\partial e}{\partial \t}=\frac{\partial Q}{\partial
\t}=c_\upsilon(\t),
\end{equation*}
where $c_\upsilon(\t)$ denotes the specific heat and $Q=Q(\t)$ is a function of $\t$.
Thus, the constitutive relation \eqref{18} implies that the internal
energy $e$ can be decomposed as a sum:
\begin{equation} \label{19}
e(\r,\t)=P_e(\r)+Q(\t),
\end{equation}
where
\begin{equation*}
P_e(\r)=\int_1^{\r}\f{p_e(\xi)}{\xi^2}d\xi,\quad Q(\t)=\int_0^{\t}c_\upsilon(\xi)d\xi.
\end{equation*}

If the flow is smooth,  multiplying equation \eqref{12} by $\u$ and  \eqref{14} by $\H$, and summing them together, we obtain
\begin{equation}\label{111}
\begin{split}
&\f{d}{dt}\left(\f{1}{2}\r|\u|^2+\f{1}{2}|\H|^2\right)+\Dv\left(\f{1}{2}\r|\u|^2\u\right)+\nabla
p\cdot\u\\& =\Dv\Psi\cdot
\u+(\nabla\times\H)\times\H\cdot\u+\nabla\times(\u\times\H)\cdot\H-\nabla\times(\nu\nabla\times\H)\cdot\H.
\end{split}
\end{equation}
Subtracting \eqref{111} from \eqref{13}, we obtain the internal
energy equation:
\begin{equation}\label{112}
\partial_t (\r e)+\Dv(\r\u e)+(\Dv\u)p=\nu|\nabla\times\H|^2+\Psi:\nabla\u+\Dv(\kappa\nabla\theta),
\end{equation}
using
$$\Dv(\nu\H\times(\nabla\times\H))=\nu|\nabla\times\H|^2-\nabla\times(\nu\nabla\times\H)\cdot\H,$$
and
\begin{equation}\label{113}
\Dv((\u\times\H)\times\H)=(\nabla\times\H)\times\H\cdot\u+\nabla\times(\u\times\H)\cdot\H,
\end{equation}
where $\Psi:\nabla\u$ denotes the scalar product of two matrices (see \eqref{matrix}).
Multiplying equation \eqref{11} by $(\r P_e(\r))'$ yields
\begin{equation}\label{114}
\partial_t(\r P_e(\r))+\Dv(\r P_e(\r)\u)+p_e(\r)\Dv\u=0,
\end{equation}
and subtracting this equality from \eqref{112}, we get the following
thermal energy equation:
\begin{equation}\label{115}
\partial_t (\r Q(\t))+\Dv(\r Q(\t)\u)-\Dv(\kappa(\t)\nabla\t)=\nu|\nabla\times\H|^2+\Psi:\nabla\u-\t
p_{\t}(\r)\Dv\u.
\end{equation}

We note that in \cite{f6}, Ducomet and Feireisl studied, using the
entropy method, the full compressible MHD equations with an
additional {Poisson's} equation under the assumption that the
viscosity coefficients depend on the temperature and the magnetic
field, and the pressure behaves like the power law $\r^\gamma$
with $\gamma=\f{5}{3}$ for large density. We also remark that, for
the mathematical analysis of incompressible MHD equations, we
refer the reader to the work \cite{gb} and the references cited
therein; and for the related studies on the multi-dimensional
compressible Navier-Stokes equations, we refer to \cite{f1, f2,
hoff97,p2} and particularly \cite{f1, f2} for the nonisentropic
case. In this paper, we consider compressible MHD flow with more
general pressure, and use the thermal equation \eqref{115} as in
\cite{f1} instead of the entropy equation used in \cite{f6}, thus
the methods of this paper differ significantly from those in
\cite{f6}. There are several major difficulties in studying the
global solutions of the initial-boundary value problem of
\eqref{1}-\eqref{17} with large data, due to the interaction from
the magnetic field, large oscillations and concentrations of
solutions, and poor \textit{a priori} estimates available for MHD.
To deal with the possible density oscillation, we use the weak
continuity property of the effective viscous flux , first
established by Lions \cite{p2} for the barotropic compressible
Navier-Stokes system with constant viscosities (see also Feireisl
\cite{f2} and Hoff \cite{hoff95}). More precisely, for fixed
$T>0$, assuming
\begin{equation*}
\begin{cases}
(\r_n, b(\r_n), p_n)\rightarrow (\r, \overline{b(\r)}, \overline{p})
\textrm{  weakly in }  L^1(\O\times(0,T)),\\
(\u_n, \H_n) \rightarrow (\u, \H) \textrm{  weakly in } L^2([0,T];W_0^{1,2}(\O)),
\end{cases}
\end{equation*}
we will prove that, for some function $b$,
\begin{equation*}
\big(p_n-(\lambda(\t_n)+2\mu(\t_n))\Dv\u_n\big)b(\r_n)\rightarrow\big(\overline{p}-
(\overline{\lambda(\t)}+2\overline{\mu(\t)})\Dv\u\big)b(\r)
\end{equation*}
weakly in $L^1(\O\times(0,T))$, where $\overline{f}$ denote a weak
limit of a sequence $\{f_n\}_{n=1}^\infty$ in
$L^1(\O\times(0,T))$. To overcome the difficulty from the
concentration in the temperature in order to pass to limit in
approximation solutions, we use the renormalization of the thermal
energy equation \eqref{115}. More precisely, multiplying
\eqref{115} by $h(\t)$ for some function $h$, we obtain,
\begin{equation}\label{117}
\begin{split}
&\partial_t (\r Q_h(\t))+\Dv(\r Q_h(\t)\u)-\Delta K_h(\t)\\&=
\nu|\nabla\times\H|^2h(\t)+h(\t)\Psi:\nabla\u-h(\t)\t p_{\t}(\r)\Dv\u
-h'(\t)\kappa(\t)|\nabla \t|^2,
\end{split}
\end{equation}
where
$$Q_h(\t)=\int_0^{\t}c_\upsilon(\xi)h(\xi)d\xi, \quad K_h(\t)=\int_0^{\t}\kappa(\xi)h(\xi)d\xi.$$
The idea of renormalization was used in  Feireisl \cite{f1,f2}, and is similar
to that in DiPerna and Lions \cite{dl}.
In addition, we also need to overcome the difficulty arising from
the presence of the magnetic field and its coupling and interaction with the fluid variables.

We organize the rest of this paper as follows. In Section 2, we
introduce a variational formulation of the full compressible MHD
equations, and also state the main existence result
(Theorem \ref{mt}). In Section 3, we will formally derive a series
of \textit{a priori} estimates on the solution. In order to
construct a sequence of approximation solutions, a three-level
approximation scheme from \cite{hw1} for isentropic MHD flow  will be  adopted in
Section 4. Finally, in Section 5, our main result will be proved
through a vanishing viscosity and vanishing artificial pressure
limit passage using the weak convergence method.

\bigskip

\section{Variational Formulation and Main Result}

In this section, we give the definition of the variational solution
to the initial-boundary value problem \eqref{1}-\eqref{17} and state the main result.

First we remark that, as shown later, the optimal estimates we can
expect on the magnetic field $\H$ and the velocity $\u$ are in
$H^1$-norms, which can not ensure the convergence of the terms
$|\nabla\times\H|^2$ and $\Psi:\nabla\u$ in $L^1$ of equation
\eqref{115}, or even worse, in the sense of distributions. In
other words, the compactness on the temperature does not seem to
be sufficient to pass to the limit in the thermal energy equation.
Thus, we will replace the thermal energy equality \eqref{115} by
two inequalities in the sense of distributions to be in accordance
with the second law of thermodynamics. More precisely, instead of
\eqref{115}, we only require that the following two inequalities
hold:
\begin{equation}\label{21}
\partial_t (\r Q(\t))+\Dv(\r Q(\t)\u)-\Delta K(\t)\geq\nu|\nabla
\times\H|^2+\Psi:\nabla\u-\t p_{\t}(\r)\Dv\u,
\end{equation}
in the sense of distributions, and
\begin{equation}\label{22}
E[\r,\u,,\t,\H](t)\leq E[\r,\u,\t,\H](0) \textrm{ for } t\geq 0,
\end{equation}
with the total energy
$$E[\r,\u,\t,\H]=\int_{\O}
\left(\r\left(P_e(\r)+Q(\t)+\f{1}{2}|\u|^2\right)+\f{1}{2}|\H|^2\right)dx,$$
and
$$K(\t)=\int_0^{\t}\kappa(\xi)d\xi.$$

Now we  give the definition of variational solutions to the full
MHD equations as follows.

\begin{Definition}\label{d1}
A vector $(\r, \u, \t, \H)$ is said to be a variational solution to
the initial-boundary value problem \eqref{1}-\eqref{17} of
the full compressible MHD equations on the time interval $(0,T)$ for any fixed $T>0$ if
the following conditions hold:
\begin{itemize} 
\item The density $\r\ge 0$, the  velocity $\u\in L^2([0,T];
W_0^{1,2}(\O))$, and the magnetic field $\H\in L^2([0,T];
W_0^{1,2}(\O))\cap C([0,T]; L_{weak}^2(\O))$ satisfy the equations
\eqref{11}, \eqref{12}, and  \eqref{14} in the sense of
distributions, and
$$\int_0^T\!\!\!\!\int_{\O}
\left(\r\partial_t\varphi+\r\u\cdot\nabla\varphi\right) dxdt=0,$$
for any $\varphi\in C^\infty(\O\times[0,T])$ with $\varphi(x,0)=\varphi(x,T)=0$ for $x\in\O$;

\item The temperature $\t$ is a non-negative function satisfying
\begin{equation*}
\begin{split}
&\int_0^T\!\!\!\!\int_{\O}\left(\r Q(\t)\partial_t\varphi+\r Q(\t)\u\cdot\nabla\varphi+K(\t)\Delta\varphi\right)dxdt\\
&\leq \int_0^T\!\!\!\!\int_{\O}(\t
p_{\t}(\r)\Dv\u-\nu|\nabla\times\H|^2-\Psi:\nabla\u)\varphi
\,\,dxdt
\end{split}
\end{equation*}
for any $\varphi\in C_0^\infty(\O\times(0,T))$ with $\varphi\geq 0$;

\item The energy inequality \eqref{22} holds for a.e $t\in(0,T)$, with
$$E[\r,\u,\t,\H](0)=\int_{\O}\left(\r_0P_e(\r_0)+\r_0Q(\t_0)+\f{1}{2}\f{|\m_0|^2}{\r_0}+\f{1}{2}|\H_0|^2\right)\,dx;$$

\item The functions $\r$, $\r\u$, and $\H$ satisfy the initial
conditions in the following weak sense:
$$\mathrm{ess}\lim_{t\to 0^+}
\int_\O(\r, \r\u, \H)(x,t)\eta(x)\,dx =\int_\O(\r_0, m_0,
\H_0)\eta\,dx,$$ 
for any $\eta\in \mathcal{D}(\O):= C_0^\infty(\O)$.
\end{itemize}
\end{Definition}

Now we are ready to state the main result of this paper.
\begin{Theorem}\label{mt}
Let $\O\subset \R^3$ be a bounded domain of class $C^{2+\tau}$ for
some $\tau>0$. Suppose that the following conditions hold:
the pressure $p$ is given by the equation \eqref{18} where  $p_e$, $p_{\t}$ are $C^1$
functions on $[0, \infty)$ and
\begin{equation}\label{23}
\begin{cases}
p_e(0)=0, \quad p_{\t}(0)=0,\\
p_e'(\r)\geq a_1\r^{\gamma-1}, \quad p'_{\t}(\r)\geq 0 \quad \textrm{for all }\r>0,\\
p_e(\r)\leq a_2\r^\gamma, \quad p_{\t}(\r)\leq a_3(1+\r^{\frac{\gamma}3}) \quad \textrm{for all } \r\geq 0, 
\end{cases}
\end{equation}
with some constants $\gamma>\f{3}{2}$,
 $a_1>0$, $a_2> 0$, and $a_3>0$;
 $\kappa=\kappa(\t)$ is a $C^1$ function on $[0,\infty)$ such that
\begin{equation}\label{25}
\underline{\kappa}(1+\t^\alpha)\leq \kappa(\t)\leq \overline{\kappa}(1+\t^\alpha),
\end{equation}
for some constants $\alpha> 2$, $\underline{\kappa}>0$, and $\overline{\kappa}>0$;
the viscosity coefficients $\mu$ and $\lambda$ are $C^1$  functions of $\t$  and globally Lipschitz on $[0,\infty)$ satisfying
\begin{equation} \label{18b}
0<\underline{\mu}\leq\mu(\t)\leq\overline{\mu},\quad
0\leq\lambda(\t)\leq\overline{\lambda},
\end{equation}
for some positive constants $\underline{\mu}$, $\overline{\mu}$,
$\overline{\lambda}$; $\nu>0$ is a constant; there exist two
positive constants $\underline{c_\upsilon}$,
$\overline{c_\upsilon}$ such that
\begin{equation}\label{26}
0<\underline{c_\upsilon}\leq c_\upsilon(\t)\leq\overline{c_\upsilon};
\end{equation}
and finally,  the initial data satisfy
\begin{equation} \label{27}
\begin{cases}
\r_0\in L^{\gamma}(\O), \quad \r_0\geq 0   \textrm{ on } \O,  \\
\t_0\in L^{\infty}(\O), \quad \t_0\geq \underline{\t}>0 \textrm{ on } \O,\\
\f{|m_0|^2}{\r_0}\in L^1(\O),\\
\H_0\in L^2(\O), \quad \Dv \H_0=0 \text{ in } \mathcal{D}'(\O).
\end{cases}
\end{equation}
Then, the initial-boundary value problem \eqref{1}-\eqref{17} of
the full compressible MHD equations has a
variational solution $(\r, \u, \t, \H)$ on $\O\times(0,T)$ for any
given $T>0$, and
\begin{equation} \label{28}
\begin{cases}
\r\in L^{\infty}([0,T]; L^{\gamma}(\O))\cap C([0,T]; L^1(\O)),\\
\u\in L^2([0,T];W_0^{1,2}(\O)), \quad \r\u\in C([0,T]; L_{weak}^{\f{2\gamma}{\gamma+1}}(\O)),\\
\t\in L^{\alpha+1}(\O\times(0,T)), \quad\r Q(\t)\in L^{\infty}([0,T]; L^1(\O)),\\
\t p_{\t}\in L^2(\O\times(0,T)), \quad\r Q(\t)\u\in L^1(\O\times(0,T)),\\
\ln(1+\t)\in L^2([0,T]; W^{1,2}(\O)), \quad\t^{\f{\alpha}{2}}\in L^2([0,T]; W^{1,2}(\O)),\\
\H\in L^2([0,T];W_0^{1,2}(\O))\cap C([0,T]; L^2_{weak}(\O)).
\end{cases}
\end{equation}
\end{Theorem}

\medskip
\begin{Remark}
In addition, the solution constructed in Theorem \ref{mt} will
satisfy the continuity equation in the sense of renormalized
solutions, that is, the integral identity
$$\int_0^T\!\!\!\!\int_{\O}\left(b(\r)\partial_t\varphi+b(\r)\u
\cdot\nabla\varphi+(b(\r)-b'(\r)\r)\Dv\u\,\varphi\right)dxdt=0$$
holds for any
$$b\in C^1[0,\infty),\quad |b'(z)z| \leq cz^{\f{\gamma}{2}} \textrm{ for z larger than some
positive } z_0.$$
and any test function
$\varphi\in C^\infty([0,T]\times\ov{\O})$ with  $\varphi(x,0)=\varphi(x,T)=0$
for $x\in\O$.
\end{Remark}

\begin{Remark}
The growth restrictions imposed on $\kappa, \mu, \lambda$, and
$c_\upsilon$ may  not be optimal, and $\gamma>\f{3}{2}$ is a
necessary condition to ensure the convergence of nonlinear term
$\r\u\otimes\u$ in the sense of distributions. In particular, our
result includes the case of constant viscosity coefficients, with the assumption
that the coefficient $\lambda\geq 0$.
\end{Remark}

\begin{Remark}
Our method also works for the case with nonzero
external force $f$ in the momentum equation. As it is obvious that in our analysis the
presence of the external force  does not
add any additional difficulty, and usually can be dealt with by
using classical Young's inequality under suitable assumptions on the
integrability of the external force $f$.
\end{Remark}

\bigskip

\section{A Priori Estimates}

To prove Theorem \ref{mt}, we first need to obtain sufficient \textit{a
priori} estimates on the solution. The total energy
conservation \eqref{22} implies
\begin{equation}\label{31}
\begin{split}
&\int_{\O}\left(\r\left(P_e(\r)+Q(\t)+\f{1}{2}|\u|^2\right)+\f{1}{2}|\H|^2\right)dx\\
&\leq\int_{\O}\left(\r_0P_e(\r_0)+\r_0Q(\t_0)+\f{1}{2}\f{|\m_0|^2}{\r_0}+\f{1}{2}|\H_0|^2\right)\,dx.
\end{split}
\end{equation}
But the assumption \eqref{23} implies that there is a positive
constant c such that
$$\r P_e(\r)\geq c\r^\gamma, \textrm{ for any } \r\geq 0.$$
Thus, \eqref{31} implies that $\r^\gamma, \r Q(\t), \f{1}{2}\r|\u|^2$ and
$\f{1}{2}|\H|^2$ are bounded in $L^\infty([0,T]; L^1(\O))$. Hence,
$$\r\in L^\infty([0,T]; L^\gamma(\O)),\quad
\r\u\in L^\infty([0,T]; L^{\f{2\gamma}{\gamma+1}}(\O)).$$

Next, in order to obtain estimates on the temperature, we introduce the entropy
$$s(\r,\t)=\int_1^{\t}\f{c_{\upsilon}(\xi)}{\xi}d\xi-P_{\t}(\r),
\quad\text{with } P_{\t}(\r)=\int_1^{\r}\f{p_{\t}(\xi)}{\xi^2}d\xi.$$
If the flow is smooth and the temperature is strictly positive, then
by direct calculation, using \eqref{11} and \eqref{115}, we obtain
\begin{equation}\label{32}
\partial_t(\r s)+\Dv(\r s\u)+\Dv\left(\f{q}{\t}\right)=
\f{1}{\t}\left(\nu|\nabla\times\H|^2+\Psi:\nabla\u\right)-\f{q\cdot\nabla\t}{\t^2}.
\end{equation}
Integrating \eqref{32}, we get
\begin{equation}\label{33}
\begin{split}
&\int_0^T\!\!\!\!\int_{\O}\left(\f{1}{\t}\left(\nu|\nabla\times\H|^2
+\Psi:\nabla\u\right)+\f{\kappa(\t)|\nabla\t|^2}{\t^2}\right)dxdt \\
&=\int_{\O}\r s(x,t)\,dx-\int_{\O}\r s(x,0)\,dx.
\end{split}
\end{equation}
On the other hand, assumptions \eqref{23} imply, using
Young's inequality,
\begin{equation}\label{34}
|\r P_{\t}(\r)|\leq c+\r P_{e}(\r) \quad\textrm{for some } c>0.
\end{equation}
Moreover, we have
\begin{equation}\label{35}
\r \int_1^{\t}\f{c_\upsilon(\xi)}{\xi}d\xi\leq \r Q(\t) \quad\textrm{for all }
\t> 0, \; \r\geq 0,
\end{equation}
since
$$\int_1^\t\f{c_\upsilon(\xi)}{\xi}d\xi \leq 0, \quad\textrm{if } 0<\t\leq 1,$$
and
$$\int_1^\t\f{c_\upsilon(\xi)}{\xi}d\xi \leq \int_1^\t c_\upsilon(\xi)d\xi=Q(\t)-Q(1)\leq Q(\t), \quad\textrm{if } \t>1.$$
Assuming that $\r s(\cdot,0)\in L^1(\O)$, then from \eqref{33}-\eqref{35} ,
using the assumption \eqref{25} and the estimates from \eqref{31},
we get
$$\int_0^T\!\!\!\!\int_{\O}|\nabla\t^{\f{\alpha}{2}}|^2+|\nabla \ln\t|^2\,dxdt\leq C,$$
which, combining the Sobolev's imbedding theorem, implies
\begin{equation}\label{36}
\ln\t \textrm{ and } \t^{\f{\alpha}{2}} \textrm{ are bounded in }
L^2([0,T]; W^{1,2}(\O)).
\end{equation}

Finally, we turn to the estimates on the velocity and the magnetic
field. Indeed, integrating \eqref{115} over $\O\times(0,T)$, we get
\begin{equation}\label{37}
\begin{split}
&\int_0^T\!\!\!\!\int_{\O}
\left(\Psi:\nabla\u+\nu|\nabla\times\H|^2\right)dxdt \\
&= \int_0^T\!\!\!\!\int_{\O}\t
p_{\t}(\r)\Dv\u\,dxdt+\int_{\O}\r Q(\t)(x,T)dx-\int_{\O}\r Q(\t)(x,0)\,dx.
\end{split}
\end{equation}
Noticing that, using H\"{o}lder inequality, one has
\begin{equation}\label{38}
\| \t p_{\t}(\r)\|_{L^2(\O)}\leq \|\t\|_{L^6(\O)}\| p_{\t}(\r)\|_{L^3(\O)}.
\end{equation}
Thus, from the assumption \eqref{23} and estimate \eqref{36}, we have
$$\t p_{\t}(\r)\in L^2(\O\times(0,T)).$$
The relation \eqref{37} together with \eqref{31}, \eqref{38}, gives rise to the estimate
$$\int_0^T\!\!\!\!\int_{\O}\left(\Psi:\nabla \u+\nu|\nabla\times\H|^2\right)dxdt\leq C(\r_0, \u_0, \t_0, \H_0).$$
The assumption \eqref{18b}, the fact
$\|\nabla\times\H\|_{L^2}=\|\nabla\H\|_{L^2}$ when $\Dv\H=0$, and
Sobolev's imbedding theorem give that
\begin{equation*}
\u, \H \textrm{ are bounded in } L^2([0,T]; W_0^{1,2}(\O)).
\end{equation*}

In summary, if $\r s(\cdot,0)\in L^1(\O)$, the system \eqref{11},
\eqref{12}, \eqref{14}, \eqref{115} with the initial-boundary
conditions \eqref{17} and our assumptions
\eqref{23}-\eqref{27} yield the following estimates:
\begin{equation}\label{310}
\begin{cases}
\r P_e(\r), \r Q(\t) \textrm{ are bounded in } L^\infty([0,T]; L^1(\O));\\
\r \textrm{ is bounded in } L^\infty([0,T]; L^\gamma(\O)), \\
\r\u\textrm{ is bounded in } L^\infty([0,T]; L^{\f{2\gamma}{\gamma+1}}(\O));\\
\ln\t \textrm{ and } \t^{\f{\alpha}{2}} \textrm{ are bounded in } L^2([0,T]; W^{1,2}(\O));\\
\u, \H \textrm{ are bounded in } L^2([0,T]; W_0^{1,2}(\O)).
\end{cases}
\end{equation}

\bigskip

\section{The Approximation Scheme and Approximation Solutions}

Similarly as Section 4 in \cite{hw1} and Section 3 in \cite{f1},
we introduce an approximate problem which consists of a system of
regularized equations:
\begin{equation}\label{41}
\begin{cases}
\rho_t +\Dv(\r\u)=\varepsilon\Delta\r, \\
(\r\u)_t+\Dv\left(\r\u\otimes\u\right)+\nabla
p(\r,\t)+\delta\nabla \r^\beta+\varepsilon\nabla\u\cdot\nabla\r
=(\na \times \H)\times \H+\Dv\Psi, \\
\begin{aligned}
\partial_t ((\r+\delta) Q(\t))+&\Dv(\r Q(\t)\u)-\Delta K(\t)+\delta\t^{\alpha+1}\\
&=(1-\delta)(\nu|\nabla\times\H|^2+\Psi:\nabla\u)-\t p_{\t}(\r)\Dv\u,
\end{aligned}\\
\H_t-\nabla\times(\u\times\H)=-\nabla\times(\nu\nabla\times\H),\quad
\Dv\H=0,
\end{cases}
\end{equation}
with the initial-boundary conditions
\begin{equation} \label{41b}
\begin{cases}
\nabla\r\cdot {\bf n}|_{\partial\O}=0, \quad \r|_{t=0}=\r_{0,\delta},\\
\u|_{\partial\O}=0, \quad \r\u|_{t=0}=m_{0,\delta},\\
\nabla\t|_{\partial\O}=0, \quad\t|_{t=0}=\t_{0,\delta},\\
\H|_{\partial\O}=0, \quad \H|_{t=0}=\H_0.
\end{cases}
\end{equation}
where $\varepsilon$ and $\delta$ are two positive parameters, $\beta>0$ is a fixed constant, and ${\bf n}$ is the unit outer normal of $\partial\O$.
The initial data are chosen in such a way that
\begin{equation}\label{42}
\begin{cases}
\r_{0,\delta}\in C^3(\overline{\O}), \quad 0<\delta\leq\r_{0,\delta}\leq\delta^{-\f{1}{2\beta}};\\
\r_{0,\delta}\rightarrow\r_0 \textrm{ in } L^\gamma(\O), \quad
|\{\r_{0,\delta}<\r_0\}|\rightarrow 0, \quad\text{as }\delta\to 0;\\
\delta\int_{\O}\r_{0,\dl}^\beta\,\,dx\rightarrow 0, \quad\text{as }\delta\to 0;\\
m_{0,\delta}=\begin{cases}m_0, &\textrm{ if } \r_{0,\delta}\geq \r_0,\\
0, &\textrm{ if } \r_{0,\delta}< \r_0;
\end{cases}\\
\t_{0,\delta}\in C^3(\overline{\O}), \quad 0<\underline{\t}\leq\t_{0,\delta}\leq\overline{\t};\\
\t_{0,\delta}\rightarrow\t_0 \textrm{ in } L^1(\O) \quad\text{as }\delta\to 0.
\end{cases}
\end{equation}

Noticing that the terms $\nu|\nabla\times\H|^2$ and $\Psi:\nabla\u$
are nonnegative, and $\t=0$ is a subsolution of the third equation in
\eqref{41}, we can conclude that, using the maximum principle,
$\t(t,x)\geq 0$ for all $t\in(0,T)$ and $x\in\O$.

From Lemma 3.2 in \cite{hw1} and Proposition 7.2 in \cite{f2}, we see that
the approximate problem \eqref{41}-\eqref{41b} with fixed positive parameters $\varepsilon$
and $\delta$ can be solved by means of a modified Faedo-Galerkin
method (cf. Chapter 7 in \cite{f2}). Thus, we state  without proof
 the following result (cf. Proposition 3.1 in \cite{f1}):

\begin{Proposition}\label{p1}
Under the hypotheses of Theorem ~\ref{mt}, and let $\beta$ be large
enough, then  the approximate problem \eqref{41}-\eqref{41b} has a solution
 $(\r, \u, \t, \H)$ on $\O\times(0,T)$ for any fixed $T>0$ satisfying the following
 properties:
\begin{itemize}
\item $\r\ge 0$, $\u\in L^2([0,T];W_0^{1,2}(\O))$, $\H\in L^2([0,T];W_0^{1,2}(\O))$,
the first equation in \eqref{41} is satisfied a.e
on $\O\times(0,T)$, the second and fourth equations in \eqref{41} are satisfied in the sense of distributions on $\O\times(0,T)$ (denoted by $\mathcal{D}'(\O\times(0,T)$),
$\u$ and $\H$ are bounded in $L^2([0,T]; W^{1,2}_0(\O))$, and, for some $r>1$,
$$\r_t, \,\Delta\r\in L^r(\O\times(0,T)), \quad
\r\u\in C([0,T]; L_{weak}^{\f{2\gamma}{\gamma+1}}(\O)),$$
\begin{equation}\label{412}
\delta\int_0^T\!\!\!\!\int_{\O}\r^{\beta+1}\,dxdt\leq c_1(\varepsilon,\delta),\quad
\varepsilon\int_0^T\!\!\!\!\int_{\O}|\nabla\r|^2\,dxdt\leq c_2,
\end{equation}
where $c_2$ is a constant independent of $\varepsilon$.

\item The energy inequality
\begin{equation*}
\begin{split}
&\int_0^T\!\!\!\!\int_{\O}(-\psi_t)\left(\f{1}{2}\r|\u|^2+\f{1}{2}|\H|^2+
\f{a_2}{\gamma-1}\r^\gamma+\f{\delta}{\beta-1}\r^\beta+(\r+\delta)Q(\t)\right)dxdt\\
&\quad +\delta\int_0^T\!\!\!\!\int_{\O}
 \psi\left(\Psi:\nabla\u+\nu|\nabla\times\u|^2+\t^{\alpha+1}\right)dxdt\\
&\leq\int_{\O}\left(\f{1}{2}\f{|m_{0,\delta}|^2}{\r_{0,\delta}}+\f{1}{2}|\H_0|^2+ \f{a_2}
{\gamma-1}\r_{0,\delta}^\gamma
+\f{\delta}{\beta-1}\r_{0,\delta}^\beta+(\r_{0,\delta}+\delta)Q(\t_{0,\delta})\right)dx\\
&\quad +\int_0^T\!\!\!\!\int_{\O}\psi p_b(\r)\Dv\u\,dxdt
\end{split}
\end{equation*}
holds for any $\psi\in C^\infty([0,T])$ satisfying
$$\psi(\cdot,0)=1, \quad\psi(\cdot,T)=0, \quad\psi_t\leq 0 \text{ on } \O,$$
where, $p_e(\r)$ has been decomposed as
$$p_e(\r)=a_2\r^\gamma-p_b(\r),$$
with $p_b\in C^1[0,\infty), \, p_b\geq 0$;

\item The temperature $\t\ge 0$ satisfies that
$$\t\in L^{\alpha+1}(\O\times(0,T)), \quad\t^{\f{\alpha}{2}}\in L^2([0,T];W^{1,2}(\O)),$$
and the thermal energy inequality holds in the following renormalized sense:
\begin{equation}\label{43}
\begin{split}
&\int_0^T\!\!\!\!\int_{\O}\left((\r+\delta)Q_h(\t)\partial_t\varphi+\r Q_h(\t)\u
\cdot\nabla\varphi+K_h(\t)\Delta\varphi-\delta h(\t)\t^{\alpha+1}\varphi\right)dxdt\\
&\leq\int_0^T\!\!\!\!\int_{\O}\left((\delta-1)h(\t)(\Psi:\nabla\u+
\nu|\nabla\times\H|^2)+h'(\t)\kappa(\t)|\nabla\t|^2\right)\varphi
\,dxdt\\
&\quad +\int_0^T\!\!\!\!\int_{\O}h(\t)\t p_{\t}(\r)\Dv\u\varphi\,dxdt-
\int_{\O}(\r_{0,\delta}+\delta)Q_h(\t_{0,\delta})\varphi(x,0)\,dx\\
&\quad +\varepsilon\int_0^T\!\!\!\!\int_{\O}\nabla\r\cdot
 \nabla\big((Q_h(\t)-Q(\t)h(\t))\varphi\big)dxdt,
\end{split}
\end{equation}
for any  function $h\in C^\infty(\R^+)$ satisfying
\begin{equation}\label{44}
\begin{split}
h(0)>0,\quad &h \textrm{ non-increasing on } [0,\infty),
\quad \lim_{\xi\rightarrow\infty}h(\xi)=0,\\ &h''(\xi)h(\xi)\geq 2(h'(\xi))^2,
\textrm{ for all } \xi\geq 0,
\end{split}
\end{equation}
and any test function $\varphi\in C^2(\O\times[0,T])$ satisfying
$$\varphi\geq 0,\quad \varphi(\cdot,T)=0,
\quad \nabla\varphi\cdot {\bf n}|_{\partial\O}=0.$$
\end{itemize}
\end{Proposition}

\begin{Remark}\label{r1}
In fact, we have
$$\|\nabla\t^{\f{\alpha}{2}}\|_{L^2(\O\times(0,T))}\leq c(\delta).$$
\end{Remark}

Comparing with the term $\nu|\nabla\times\H|^2$ in \eqref{43}, we
need to pay more attention to the term $\Psi:\nabla\u$, because the
later involves temperature-dependent coefficients and thus can not
be dealt with by the standard weak lower semi-continuity.
 Indeed, the hypothesis \eqref{44} was imposed in
\cite{f2} in order to make the function
$$(\t,\nabla\u)\mapsto h(\t)\Psi:\nabla\u$$
convex, and, consequently, weakly lower semi-continuous (the stress
tensor $\Psi$ in \cite{f2} depends on $\nabla\u$ only). In
accordance with our \textit{new} context, the following lemma is
useful:

\begin{Lemma}\label{l41}
Let $g(\t)$ be a bounded, continuous and non-negative function from
$[0,\infty)$ to $\R$. Suppose that $\t_n$ and $\u_n$ are two sequences of functions defined on $\O$ and
$$\t_n\rightarrow \t \textrm{ a.e in } \O,$$
and
$$\u_n\rightarrow\u \textrm{ weakly in } W^{1,2}(\O).$$
 Then,
\begin{equation}\label{45}
\int_{\O}g(\t)h(\t)|\nabla\u|^2\,dx\leq
\liminf_{n\rightarrow\infty}\int_{\O}g(\t_n)h(\t_n)|\nabla\u_n|^2\,dx.
\end{equation}
In particular,
\begin{equation}\label{46}
\int_{\O}h(\t)\Psi:\nabla\u\,dx\leq
\liminf_{n\rightarrow\infty}\int_{\O}h(\t_n)\Psi(\u_n):\nabla\u_n\,dx.
\end{equation}
\end{Lemma}

\begin{proof}
First we show that $\sqrt{g(\t_n)}\nabla\u_n$ converges weakly to
$\sqrt{g(\t)}\nabla\u$ in $L^2$. Indeed, since $\sqrt{g(\t_n)}
\nabla\u_n$ is uniformly bounded in $L^2$, it is enough to show
\begin{equation}\label{47}
\int_{\O}\sqrt{g(\t_n)}\nabla\u_n
\phi\,dx\rightarrow\int_{\O}\sqrt{g(\t)}\nabla\u \,\phi\,dx, \textrm{
for all } \phi\in C^\infty(\O).
\end{equation}
Since $\t_n\rightarrow \t \textrm{ a.e in } \O$,
then $\sqrt{g(\t_n)}\phi\rightarrow \sqrt{g(\t)}\phi \textrm{ a.e in }
\O$ for all $\phi\in C^\infty(\O)$. Thus by Lebesgue's dominated
convergence theorem, we know that
$$\sqrt{g(\t_n)}\phi\rightarrow \sqrt{g(\t)}\phi \textrm{ in } L^2(\O),$$  and \eqref{47} follows.

Next, by virtue of Corollary 2.2 in \cite{f2}, it is enough to
observe that the function $\Phi: (\t,\xi)\mapsto h(\t)\xi^2$ is convex
and continuous on $\R^{+}\times\R$. Computing the Hessian
matrix of $\Phi$, we get
$$\textrm{det}\{\partial^2_{\t,\xi}\Phi\}=2\xi^2(h''(\t)h(\t)-2(h'(\t))^2)\geq 0,$$
and
$$\textrm{trace}\{\partial^2_{\t,\xi}\Phi\}=\xi^2h''(\t)+2h(\t)\geq 0,$$
provided $\t>0$ and $h$ satisfies \eqref{44}. Thus the Hessian matrix
is positively definite; therefore $\Phi$ is convex and continuous.
Then, \eqref{45} is a direct application of Corollary 2.2 in
\cite{f2}.

Finally, from \eqref{45} and the calculation:
\begin{equation} \label{matrix}
\Psi:\nabla\u=\sum^3_{i,j=1}\f{\mu(\t)}{2}\left(\f{\partial u^i}{\partial x_j}
+\f{\partial u^j}{\partial x_i}\right)^2+\lambda(\t)|\Dv\u|^2,
\end{equation}
\eqref{46} follows.
\end{proof}

In the next two steps, in order to obtain the variational solution
of the initial-boundary value problem \eqref{1}-\eqref{17}, we
need to take the vanishing limits of the artificial viscosity
$\varepsilon\to 0$ and artificial pressure coefficient $\delta\to
0$ in the approximate solutions of \eqref{41}-\eqref{41b}. As seen
in \cite{f1,f2,hw1}, the techniques used in those two procedures
are rather similar. Moreover, in some sense, the step of taking
$\varepsilon\rightarrow 0$ is much easier than the step of taking
$\delta\rightarrow 0$ due to the higher integrability of $\r$.
Hence we will omit the step of taking $\varepsilon\rightarrow 0$
(readers can refer to Section 5 in \cite{hw1} or Section 4 in
\cite{f1}), and focus on the step of taking $\delta\rightarrow 0$.
Thus, we state without proof the result as $\varepsilon\rightarrow
0$ as follows.

\begin{Proposition}\label{p2}
Let $\beta>0$ be large enough and $\delta>0$ be fixed, then the
initial-boundary value problem \eqref{1}-\eqref{17} for full
compressible MHD equations admits an approximate solution
$(\r,\u,\t,\H)$ with parameter $\delta$ (as the limit of the
solutions to \eqref{41}-\eqref{41b} when $\varepsilon\to 0$) in
the following sense:
\begin{itemize}
\item The density $\r$ is a non-negative function, and
$$\r\in C([0,T];L^\beta_{weak}(\O)),$$
satisfying the initial condition in \eqref{42}. The velocity $\u$
and the magnetic field $\H$ belong to $L^2([0,T]; W_0^{1,2}(\O))$.
The equation \eqref{11} and \eqref{14} are satisfied in
$\mathcal{D}'(\O\times(0,T))$ and
$$\delta\int_0^T\!\!\!\!\int_{\O}\r^{\beta+1}\,dxdt\leq c(\delta).$$
Moreover, $\r, \u$ also solve equation \eqref{11} in the sense of renormalized solutions;
\item The functions $\r, \u, \t, \H$ solve a modified momentum equation
\begin{equation}\label{48}
(\r\u)_t+\Dv\left(\r\u\otimes\u\right)+\nabla (p(\r,\t)+\delta\r^\beta)=(\na \times \H)\times \H+\Dv\Psi,
\end{equation}
in $\mathcal{D}'(\O\times(0,T))$.
Furthermore, the momentum
$$\r\u\in C([0,T]; L_{weak}^{\f{2\gamma}{\gamma+1}}(\O))$$ satisfies the initial condition in \eqref{42};
\item The energy inequality
\begin{equation}\label{49}
\begin{split}
&\int_0^T\!\!\!\!\int_{\O}(-\psi_t)\left(\f{1}{2}\r|\u|^2
+\f{1}{2}|\H|^2+\r P_e(\r)+\f{\delta}{\beta-1}\r^\beta+(\r+\delta)Q(\t)\right)dxdt\\
&\quad +\delta\int_0^T\!\!\!\!\int_{\O}\psi\left(\Psi:\nabla\u+\nu|\nabla\times\u|^2
   +\t^{\alpha+1}\right) dxdt \\
&\leq\int_{\O}\left(\f{1}{2}\f{|m_{0,\dl}|^2}{\r_{0,\dl}}+\f{1}{2}|\H_0|^2+\r_{0,\delta}
P_e(\r_{0,\delta})+\f{\delta}{\beta-1}\r_{0,\delta}^\beta+(\r_{0,\delta}+\delta)
 Q(\t_{0,\delta})\right)dx
\end{split}
\end{equation}
holds for any $\psi\in C^\infty([0,T])$ satisfying
$$\psi(\cdot,0)=1,\quad \psi(\cdot,T)=0,\quad \psi_t\leq 0;$$
\item The temperature $\t$ is a non-negative function, and
\begin{equation}\label{410}
\t\in L^{\alpha+1}(\O\times(0,T)), \quad\t^{\f{\alpha+1-\omega}{2}}\in L^2([0,T];W^{1,2}(\O)),\quad \omega\in(0,1],
\end{equation}
satisfying the thermal energy inequality in the following renormalized sense:
\begin{equation}\label{411}
\begin{split}
&\int_0^T\!\!\!\!\int_{\O}\left((\r+\delta)Q_h(\t)\partial_t\varphi+
\r Q_h(\t)\u\cdot\nabla\varphi+K_h(\t)\Delta\varphi-\delta h(\t)\t^{\alpha+1}\varphi\right)dxdt\\
&\leq\int_0^T\!\!\!\!\int_{\O}\left((\delta-1)h(\t)
(\Psi:\nabla\u+\nu|\nabla\times\H|^2)+h'(\t)\kappa(\t)|\nabla\t|^2\right)\varphi dxdt\\
&\quad +\int_0^T\!\!\!\!\int_{\O}h(\t)\t
p_{\t}(\r)\Dv\u\varphi\,dxdt-\int_{\O}(\r_{0,\delta}+\delta)Q_h(\t_{0,\delta})\varphi(0)\,dx,
\end{split}
\end{equation}
for any admissible function $h\in C^\infty(\R^+)$ satisfying \eqref{44}
and any test function $\varphi\in C^2(\O\times[0,T])$ satisfying
$$ \varphi\geq 0,\quad \varphi(\cdot,T)=0,\quad \nabla\varphi\cdot {\bf n}|_{\partial\O}=0.$$
\end{itemize}
\end{Proposition}

\medskip
\begin{Remark}
In Proposition \ref{p2}, the second estimate in \eqref{410} can be explained as follows. Taking
$$h(\t)=\f{1}{(1+\t)^\omega},\quad\omega\in(0,1],
\quad \varphi(t,x)=\psi(t),\quad 0\leq\psi\leq 1,\quad\psi\in \mathcal{D}(0,T),$$
in \eqref{43}, we obtain
\begin{equation*}
\begin{split}
&\omega\int_0^T\int_{\O}\f{\kappa(\t_\varepsilon)}
{(1+\t_\varepsilon)^{\omega+1}}|\nabla\t_\varepsilon|^2\psi\,dxdt\\
&\leq -\int_0^T\!\!\!\!\int_{\O}(\r_\varepsilon+
\delta)Q_h(\t_\varepsilon)\psi_t\,dxdt
+\delta\int_0^T\!\!\!\!\int_{\O} h(\t_\varepsilon)\t_\varepsilon^{\alpha+1}\psi\,dxdt\\
&\quad +\int_0^T\!\!\!\!\int_{\O}\t_\varepsilon p_{\t}
(\r_\varepsilon)|\Dv\u_\varepsilon|\psi\,dxdt\\
&\quad +\varepsilon\int_0^T\!\!\!\!\int_{\O}
|\nabla\r_\varepsilon\cdot\nabla\left((Q_h(\t_\varepsilon)-Q(\t_\varepsilon)
h(\t_\varepsilon))\varphi\right)|\,dxdt.
\end{split}
\end{equation*}
Observing that
$$\int_0^T\!\!\!\!\int_{\O}\left|\nabla(1+\t_\varepsilon)^{
\f{\alpha+1-\omega}{2}}\right|^2\psi\,dxdt\leq
c\int_0^T\!\!\!\!\int_{\O}\f{\kappa(\t_\varepsilon)}
{(1+\t_\varepsilon)^{\omega+1}}|\nabla\t_\varepsilon|^2\psi\,dxdt,$$
and
\begin{equation*}
\begin{split}
&\int_0^T\!\!\!\!\int_{\O}\t_\varepsilon p_\t(\r_\varepsilon)|\Dv\u_\varepsilon|\psi\,dxdt\\
&\leq c\|\t_\varepsilon\|_{L^2([0,T]; L^6(\O))}\|\u_\varepsilon
\|_{L^2([0,T]; W^{1,2}(\O))}\| p_\t(\r_\varepsilon)\|_{L^\infty([0,T]; L^3(\O))},
\end{split}
\end{equation*}
and, by the hypothesis \eqref{26}, we have
\begin{equation*}
\begin{split}
&\varepsilon\int_0^T\!\!\!\!\int_{\O}|\nabla\r_\varepsilon\cdot\nabla[(Q_h(\t_\varepsilon)-Q(\t_\varepsilon)h(\t_\varepsilon))\varphi]|
\,dxdt\\
&\leq c\varepsilon\|\psi\|_{L^\infty}
\|\nabla\r_\varepsilon\|_{L^2}\left\|\f{Q(\t_\varepsilon)}{(1+\t_\varepsilon)^{\omega+1}}
\nabla\t_\varepsilon\right\|_{L^2} \\
&\leq c\varepsilon\|\nabla\r_\varepsilon\|_{L^2}
\left\| (1+\t_\varepsilon)^{\f{\alpha-1-\omega}{2}}\nabla\t_\varepsilon\right\|_{L^2}\\
&=c\varepsilon\|
\nabla\r_\varepsilon\|_{L^2}\left\|\nabla(1+\t_\varepsilon)^{\f{\alpha+1-\omega}{2}}
\right\|_{L^2},
\end{split}
\end{equation*}
where, the following property is used
$$\f{Q(\t)^2}{(1+\t)^{\omega+1}}|\nabla\t|^2\leq c(1+\t)^\alpha|\nabla\t|^2.$$
By Young's inequality, Remark \ref{r1}, and the energy
inequality in Proposition \ref{p1}, one has
$$(1+\t_\varepsilon)^{\f{\alpha+1-\omega}{2}}\in L^2([0,T];
W^{1,2}(\O)).$$
 Thus, $$\t_\varepsilon^{\f{\alpha+1-\omega}{2}}\in
L^2([0,T]; W^{1,2}(\O)),$$ since
$$\t_\varepsilon^{\f{\alpha-1-\omega}{2}}|\nabla\t_\varepsilon|
\leq(1+\t_\varepsilon)^{\f{\alpha-1-\omega}{2}}|\nabla\t_\varepsilon|.$$

\end{Remark}

In the next Section, we shall take the limit of the other artificial term: the artificial pressure, as $\delta\to 0$.

\bigskip

\section{ The Limit of Vanishing Artificial Pressure}

In this section, we take the limit as $\delta\to 0$ to eliminate
the $\delta$-dependent terms appearing in \eqref{41}, while in the
previous Section passing to the limit as $\varepsilon\to 0$ has
been done. Denote by $\{\r_\delta, \u_\delta, \t_\delta,
\H_\delta\}_{\delta>0}$ the sequence of approximate solutions
obtained in Proposition \ref{p2}. In addition to the possible
oscillation effects on density, the concentration effects on
temperature is also a major issue of this section. To deal with
these difficulties, we employ a variant of well-known
Feireisl-Lions method \cite{f1,f2,p2} in our \textit{new} context.

\subsection{Energy estimates}

The main object in this subsection is to find sufficient \textit{a
priori} estimates. First, our choice of
the initial data \eqref{42} implies that, as $\delta\to 0$,
\begin{equation*}
\begin{split}
&\int_{\O}\left(\f{1}{2}\f{|m_{0,\delta}|^2}{\r_{0,\delta}}+\f{1}{2}|\H_0|^2+
\r_{0,\delta} P_e(\r_{0,\delta})+\f{\delta}{\beta-1}
\r_{0,\delta}^\beta+(\r_{0,\delta}+\delta)Q(\t_{0,\delta})\right)dx\\
&\rightarrow E[\r, \u, \t,
\H](0)=\int_{\O}\left(\r_0P_e(\r_0)+\r_0Q(\t_0)+\f{1}{2}
\f{|m_0|^2}{\r_0}+\f{1}{2}|\H_0|^2\right)dx.
\end{split}
\end{equation*}
Hence, from the energy inequality \eqref{49}, we can conclude that
\begin{equation}\label{51}
\r_\delta \quad\textrm{ is bounded in } L^\infty([0,T]; L^\gamma(\O)),
\end{equation}
\begin{equation}\label{52}
\sqrt{\r_\delta}\u_\delta, \; \H_\delta \quad\textrm{ are bounded in } L^\infty([0,T]; L^2(\O)),
\end{equation}
\begin{equation}\label{53}
(\r_\delta+\delta)Q(\t_\delta) \quad\textrm{ is bounded in } L^\infty([0,T]; L^1(\O)),
\end{equation}
\begin{equation}\label{55}
\delta\int_0^T\!\!\!\!\int_{\O}\left(\r_\delta^{\beta}+\t_\delta^{\alpha+1}\right)dxdt\leq
c,
\end{equation}
for some constant c, which is independent of $\delta$.

Now, we take $$\varphi(x,t)=\f{T-\f{1}{2}t}{T},\quad h(\t)=\f{1}{1+\t}$$
 in \eqref{411}  to obtain
\begin{equation}\label{57}
\begin{split}
&\int_0^T\!\!\!\!\int_{\O}\left(\f{1-\delta}{1+\t_\delta}
(\Psi_\delta:\nabla\u_\delta+\nu|\nabla\times\H_\delta|^2)+\f{\kappa(\t_\delta)}
{(1+\t_\delta)^2}|\nabla\t_\delta|^2\right)\,dxdt\\
&\leq
2\int_0^T\!\!\!\!\int_{\O}\f{\t_\delta}{1+\t_\delta}p_\t
(\r_\delta)\Dv\u_\delta\,dxdt+2\delta\int_0^T\!\!\!\!\int_{\O}\t_\delta^{\alpha}
\,dxdt\\
&\quad -2\int_{\O}(\r_{0,\delta}+\delta)Q_{1}(\t_{0,\delta})\,dx+\int_{\O}(\r_{\delta}+\delta)Q_{1}
(\t_{\delta})(T)\,dx,
\end{split}
\end{equation}
where
$$Q_{1}(\t)=\int_0^\t\f{c_\upsilon(\xi)}{1+\xi}d\xi \leq Q(\t).$$

Using the estimates \eqref{53} and \eqref{55}, we deduce from \eqref{57} that
\begin{equation*}
\begin{split}
&\int_0^T\!\!\!\!\int_{\O}\left(\f{1-\delta}{1+\t_\delta}
(\Psi_\delta:\nabla\u_\delta+\nu|\nabla\times\H_\delta|^2)+\f{\kappa(\t_\delta)}
{(1+\t_\delta)^2}|\nabla\t_\delta|^2\right)\,dxdt\\
&\leq
c\left(1+\int_0^T\!\!\!\!\int_{\O}p_{\t}(\r_\delta)\Dv\u_\delta\,dxdt\right),
\end{split}
\end{equation*}
for some constant c which is independent of $\delta$, and here the
second term on right-hand side can be rewritten with the help of the
renormalized continuity equation as
$$\int_0^T\!\!\!\!\int_{\O}p_{\t}(\r_\delta)\Dv\u_\delta\,dxdt=
\int_0^T\!\!\!\!\int_{\O}\partial_t(\r_\delta P_\t(\r_\delta))\,dxdt.$$
By the hypothesis \eqref{23} and the estimate \eqref{51}, one has
\begin{equation*}
\begin{split}
\int_0^T\!\!\!\!\int_{\O}\partial_t(\r_\delta P_\t(\r_\delta)) dxdt
\leq c(\O)\left(1+\int_{\O}\r_\delta^{\f{\gamma}{3}}\,dx\right)
\leq c(\O)\left(1+\int_{\O}\r_\delta^{\gamma}\,dx\right)\leq c(\O,T).
\end{split}
\end{equation*}
Consequently, we can conclude that
\begin{equation*}
\f{\kappa(\t_\delta)}
{(1+\t_\delta)^2}|\nabla\t_\delta|^2\in L^1(\O\times(0,T)),
\end{equation*}
which, combining with the hypothesis \eqref{25}, gives us that
\begin{equation*}
\nabla\ln(1+\t_\delta), \;  \nabla\t_\dl^{\f{\alpha}{2}}  \textrm{ are bounded in } L^2(\O\times(0,T)).
\end{equation*}
Thus, combining with Sobolev's imbedding theorem, we obtain that
\begin{equation}\label{59}
\ln(1+\t_\delta), \; \t_\dl^{\f{\alpha}{2}}  \textrm{ are bounded
in } L^2([0,T]; W^{1,2}(\O)).
\end{equation}
Moreover, in view of the hypothesis \eqref{23}, we get
\begin{equation}\label{510}
\t_\delta p_\t(\r_\delta) \textrm{ is bounded in } L^2(\O\times(0,T)).
\end{equation}

With \eqref{510} in hand, we can repeat the same procedure as above,
taking now $$h(\t)=\f{1}{(1+\t)^\omega},\quad \omega\in(0,1), \quad \varphi(x,t)=\f{T-\f{1}{2}t}{T},$$
 and finally we can get
\begin{equation}\label{511}
\int_0^T\!\!\!\!\int_{\O}\left(\f{1-\delta}{(1+\t_\delta)^\omega}
(\Psi_\delta:\nabla\u_\delta+\nu|\nabla\times\H_\delta|^2)+\omega\f{\kappa(\t_\delta)}
{(1+\t_\delta)^{1+\omega}}|\nabla\t_\delta|^2\right)dxdt\leq c,
\end{equation}
for some constant c which is independent of $\delta$.

Letting $\omega\rightarrow 0$ and using the monotone convergence
theorem, we deduce that
\begin{equation}\label{513}
\u_\delta, \; \H_\delta    \textrm{ are bounded in } L^2([0,T]; W^{1,2}_0(\O)).
\end{equation}
Moreover, from \eqref{511}, we have
\begin{equation}\label{514}
(1+\t_\dl)^{\f{\alpha+1-\omega}{2}}  \textrm{ is bounded in }
L^2([0,T]; W^{1,2}(\O)), \textrm{ for any } \omega\in(0,1].
\end{equation}
In particular, this implies that
\begin{equation}\label{515}
\t_\delta  \textrm{ is bounded in } L^2([0,T];L^{3\alpha+2}(\O)).
\end{equation}
Using H\"{o}lder inequality, we have
$$\|\t_\delta^{\alpha+\f{4}{3}}\|_{L^1(D)}\leq \parallel\t_\delta^{\alpha+\f{2}{3}}
\|_{L^{3}(D)}\|\t_\delta^{\f{2}{3}}\|_{L^{\f{3}{2}}(D)}\leq
\parallel
\t_\delta^{\alpha+\f{2}{3}}\|_{L^{3}(D)}\|\t_\delta\|_{L^1(D)}^{\f{2}{3}}$$
for any $D\subset\O$, which, together with \eqref{514}, \eqref{53}
and the hypothesis \eqref{26}, yields
\begin{equation*}
\int_{\{\r_\delta\geq
d\}}\t_\delta^{\alpha+\f{4}{3}}\,dxdt\leq c(d),
\end{equation*}
for any $d>0$.
In particular,
\begin{equation}\label{a1}
\int_{\{\r_\delta\geq d\}}\t_\delta^{\alpha+1}\,dxdt\leq
c(d)
\end{equation}
for any $d>0$.

\subsection{Temperature estimates}

In order to pass to the limit in the term $K(\t)$, our aim in this
subsection is to derive uniform estimates on $\t_\delta$ in
$L^{\alpha+1}(\O\times(0,T))$. To this end, we will follow the
argument in \cite{f2}.

To begin with, we have
$$\int_{\{\r_\delta\geq d\}}\r_\delta\,dx\geq M_\delta-d|\O|\geq \f{M}{2}-d|\O|,$$
where $M_\delta$ denotes the total mass,
$$M_\delta=\int_\O\r_\delta\,dx,$$
independent of $t\in[0,T]$, and
$$M=\int_\O\r_0\,dx>0.$$
On the other hand, H\"{o}lder inequality yields
$$\int_{\{\r_\delta\geq d\}}\r_\delta\,dxdt\leq
\|\r_\delta\|_{L^\gamma(\O)}|\{\r_\delta\geq d\}|^{\f{\gamma-1}{\gamma}}.$$
Consequently, there exists a function $\Lambda=\Lambda(d)$
independent of $\delta>0$ such that
$$|\{\r_\delta\geq d\}|\geq\Lambda(d)>0 \textrm{ for all } t\in[0,T], \textrm{ if }0\leq d<\f{M}{2|\O|}.$$
Fix $0<d<\f{M}{4|\O|}$ and choose a function $b\in C^\infty(R)$ such that
$$ b \textrm{ is non-increasing}; \quad b(z)=0 \textrm{ for } z\leq d,\quad b(z)=-1 \textrm{ if } z\geq 2d.$$
For each $t\in[0,T]$, let $\eta=\eta(t)$ be the unique solution of the Neumann problem:
\begin{equation}\label{a2}
\begin{cases}
\Delta\eta=b(\r_\delta(t))-\f{1}{|\O|}\int_\O b(\r_\delta(t))\,dx \, \textrm{ in }\O,\\
\nabla\eta\cdot {\bf n}|_{\partial\O}=0, \quad \int_\O\eta\,dx=0,
\end{cases}
\end{equation}
where $\eta=\eta(t)$ is a function of spatial variable $x\in\O$ with $t$ as a parameter.
Since the right-hand side of \eqref{a2} has a bound which is
independent of $\delta$, there is a constant $\underline{\eta}$ such
that
$$\eta=\eta(t)\geq \underline{\eta} \textrm{ for all } \dl>0 \text{ and } t\in[0,T].$$

Now, we take
$$\varphi(x,t)=\psi(t)(\eta-\underline{\eta}),\quad 0\leq\psi\leq 1, \quad\psi\in\mathcal{D}(0,T),$$
as a test function in \eqref{411} to obtain
\begin{equation}\label{a3}
\begin{split}
&\int_0^T\!\!\!\!\int_{\O} K_h(\t_{\delta})\left(b(\r_{\delta}(t))-
\f{1}{|\O|}\int_{\O} b(\r_{\delta}(t))\right)\psi\,dxdt\\
&\leq 2\|\eta\|_{L^\infty(\O\times(0,T))}\int_0^T\!\!\!\!\int_{\O}
\left(\dl h(\t_\dl)\t_\dl^{1+\alpha}+h(\t_\dl)\t_\dl p_\t(\r_\dl)|\Dv\u_\dl|\right)dxdt\\
&\quad + \|\nabla\eta\|_{L^\infty(\O\times(0,T))}\int_0^T\!\!\!\!\int_{\O}
\r_\dl Q_h(\t_\dl)|\u_\dl|\,dxdt\\
&\quad +\int_0^T\!\!\!\!\int_{\O}\left((\r_\dl+\dl)Q_h(\t_\dl)(\underline{\eta}-
\eta)\psi_t-(\r_\dl+\dl)Q_h(\t_\dl)\psi
\partial_t\eta\right)dxdt.
\end{split}
\end{equation}

Next, taking $$h(\t)=\f{1}{(1+\t)^\omega}$$ for $0<\omega<1$, letting
$\omega\rightarrow 0$, using Lebesgue's dominated convergence
theorem and the estimates \eqref{51}, \eqref{53}, \eqref{55},
and \eqref{510}, we obtain
\begin{equation}\label{a4}
\begin{split}
&\int_0^T\!\!\!\!\int_\O K(\t_\dl)\left(b(\r_{\delta}(t))-
\f{1}{|\O|}\int_{\O} b(\r_{\delta}(t))\right)\,dxdt\\
&\leq
c\left(1+\int_0^T\!\!\!\!\int_\O(\r_\delta+\dl)\t_\dl|\partial_t\eta|\,dxdt\right).
\end{split}
\end{equation}
On the other hand,
\begin{equation*}
\begin{split}
&\int_0^T\!\!\!\!\int_\O K(\t_\delta)\left(b(\r_{\delta}(t))-
\f{1}{|\O|}\int_{\O} b(\r_{\delta}(t))\right)\,dxdt\\
&=\int_{\{\r_\dl<d\}}K(\t_\dl)\left(b(\r_{\delta}(t))-\f{1}
{|\O|}\int_{\O} b(\r_{\delta}(t))\right)\,dxdt\\
&\quad +\int_{\{\r_\dl\geq
d\}}K(\t_\dl)\left(b(\r_{\delta}(t))-\f{1}{|\O|}\int_{\O}
b(\r_{\delta}(t))\right)\,dxdt.
\end{split}
\end{equation*}
where, by virtue of \eqref{a1}, the second integral on the
right-hand side is bounded by a constant independent of $\dl$.
Furthermore,
$$-\f{1}{|\O|}\int_{\O}b(\r_\dl)\,dx\geq -\f{1}{|\O|}\int_{\{\r_\dl
\geq 2d\} }b(\r_\dl)\,dx=\f{|\{\r_\dl\geq 2d\} |}{|\O|}\geq \f{\Lambda(2d)}{|\O|}>0.$$
Thus, we obtain
\begin{equation}\label{a5}
\begin{split}
&\int_{\{\r_\dl<d\}}K(\t_\dl)\left(b(\r_{\delta}(t))-\f{1}{|\O|}\int_{\O}
 b(\r_{\delta}(t))\right)\,dxdt\\
&\geq\f{\Lambda(2d)}{|\O|}\int_{\{\r_\dl<d\}}K(\t_\dl)\,dxdt.
\end{split}
\end{equation}
Combining  \eqref{a4}, \eqref{a5} together, we get
\begin{equation}\label{a6}
\int_{\{\r_\dl<d\}}K(\t_\dl)\,dxdt\leq
c\left(1+\int_0^T\!\!\!\!\int_\O(\r_\dl+\dl)\t_\dl|\partial_t\eta|\,dxdt\right)
\end{equation}
with c independent of $\dl$.

Finally, since $\r_\dl$ is a renormalized solution of the equation \eqref{11}, we have,
\begin{equation*}
\begin{split}
&\Delta(\partial_t\eta)=\partial_t b(\r_\dl)-\f{1}{|\O|}\int_\O\partial_t b(\r_\dl)\,dx\\
&=(b(\r_\dl)-b'(\r_\dl)\r_\dl)\Dv\u_\dl-\Dv(b(\r_\dl)\u_\dl)+\f{1}{|\O|}\int_\O
(b'(\r_\dl)\r_\dl-b(\r_\dl))\Dv\u_\dl\,dx.
\end{split}
\end{equation*}
Hence,
$$\partial_t\eta \textrm{ is bounded in } L^2([0,T];W^{1,2}(\O)).$$
Consequently, using \eqref{a6}, we conclude that
$$\int_{\{\r_\dl< d\}}K(\t_\dl)\,dxdt\leq c, \textrm{ c independent of } \dl$$
which, together with \eqref{a1} and the hypothesis \eqref{25}, yields
\begin{equation}\label{a7}
\t_\dl \textrm{ is bounded in } L^{1+\alpha}(\O\times(0,T)).
\end{equation}

\subsection{ Refined pressure estimates}

Our goal now is to improve estimates on pressure. We follow step by
step the argument of Section 5.1 in \cite{hw1}, that is, using the
Bogovskii operator ``$\Dv^{-1}[\ln(1+\r_\delta)]$'' as a test
function for the modified momentum equation \eqref{48}. Similarly to
Lemma 5.1 in \cite{hw1}, we have the estimate
\begin{equation}\label{516}
\int_{\O}(p(\r_\delta,
\t_\dl)+\delta\r_\delta^\beta)\ln(1+\r_\delta)\,dx\leq c,
\end{equation}
and hence
\begin{equation}\label{517}
p(\r_\dl, \t_\dl)\ln(1+\r_\delta) \textrm{ is bounded in }
L^1(\O\times(0,T)).
\end{equation}
Let us define the set
\begin{equation*}
J^{\dl}_k=\{(x,t)\in (0,T)\times \Omega: \r_\dl
(x,t)\leq k\} \textrm{ for } k>0 \textrm{ and } \dl \in
(0,1).
\end{equation*}
In view of \eqref{51} and the hypothesis \eqref{23}, there exists
a constant $s\in(0,\infty)$ such that for all $\dl \in (0,1)$ and
$k>0$,
$$
|\{(0,T)\times\Omega-J_k^\dl\}|\leq \f{s}{k}.
$$
We have the following estimate:
\begin{equation}\label{518}
\begin{split}
\int_0^T\!\!\!\!\int_{\O}\dl \r_\dl^\beta\,dxdt
&= \int_{J_k^\dl}\dl\r_\dl^\beta\,dxdt+
\int_{\O\times(0,T)-J_k^\dl}\dl\r_\dl^\beta\,dxdt\\
&\leq T\dl k^\beta |\O|+\dl
\int_0^T\!\!\!\!\int_{\O}\chi_{\O\times(0,T)-J_k^\dl}\r_\dl^\beta\,dxdt.
\end{split}
\end{equation}
Then, by the H\"{o}lder inequality in Orlicz spaces (cf. \cite{a1}) and the estimate
\eqref{516}, we obtain
\begin{equation}\label{519}
\begin{split}
&\dl\int_0^T\!\!\!\!\int_{\O}\chi_{\O\times(0,T)-J_k^\dl}\r_\dl^\beta\,dxdt\leq
\dl\|\chi_{\O\times(0,T)-J_k^\dl}\|_{L_N}\max\{1,\int_0^T\!\!\!\!\int_{\O}
M(\r_\dl^\beta)\,dxdt\}\\
&\leq\dl\left(N^{-1}\left(\f{k}{s}\right)\right)^{-1}\max\{1, \int_0^T\!\!\!\!\int_{\O}2
(1+\r_\dl^\beta)\textrm{ln}(1+\r_\dl^\beta)\,dxdt\}\\
&\leq \dl\left(N^{-1}\left(\f{k}{s}\right)\right)^{-1}\max\{1,
(4\textrm{ln}2)T|\O|
+4\beta\int_0^T\!\!\!\!\int_{\O\cap\{\r_\dl\geq 1\}}
\r_\dl^\beta\textrm{ln}(1+\r_\dl)\,dxdt\}\\
&\leq \left(N^{-1}\left(\f{k}{s}\right)\right)^{-1}\max\{\dl,
(4\textrm{ln}2)\dl
T|\O|+4\dl\beta\int_0^T\!\!\!\!\int_{\O}
\r_\dl^\beta\textrm{ln}(1+\r_\dl)\,dxdt\},
\end{split}
\end{equation}
where $L_M(\Omega)$, and $L_N(\Omega)$ are two Orlicz Spaces generated by two complementary N-functions
$$M(s)=(1+s)\textrm{ln}(1+s)-s,\quad N(s)=e^s-s-1,$$
respectively.
Due to \eqref{516}, we know, if $\delta<1$,
$$\max\{\dl, (4\textrm{ln}2)\dl T|\O|+4\dl\beta\int_0^T\!\!\!\!\int_{\O}
\r_\dl^\beta\textrm{ln}(1+\r_\dl)\,dxdt\}\leq c,$$ for some
$c>0$ which is independent of $\delta$.
Combining \eqref{518} with \eqref{519}, we obtain the following estimate
\begin{equation*}
\left|\int_0^T\!\!\!\!\int_{\O}\dl \r_\dl^\beta\,dxdt\right|\leq T\dl
k^\beta
|\O|+c\left(N^{-1}\left(\f{k}{s}\right)\right)^{-1},
\end{equation*}
where $c$ does not depend on $\dl$ and k. Consequently
\begin{equation}\label{520}
\limsup_{\dl\rightarrow 0}\left|\int_0^T\!\!\!\!\int_{\O}\dl
\r_\dl^\beta\,dxdt\right|\leq
c\left(N^{-1}\left(\f{k}{s}\right)\right)^{-1}.
\end{equation}
The right-hand side of \eqref{520} tends to zero as $k\rightarrow\infty$.
Thus, we have
\begin{equation*}
\lim_{\dl\rightarrow 0}\int_0^T\!\!\!\!\int_{\O}\dl
\r_\dl^\beta\,dxdt=0,
\end{equation*}
which yields
\begin{equation}\label{521}
\dl \r_\dl^\beta \rightarrow 0  \textrm{ in }
\mathcal{D}'(\O\times(0,T)).
\end{equation}

\subsection{Strong convergence of the temperature}

Since $\r_\delta$, $\u_\delta$ satisfy the continuity equation \eqref{11}, then
\begin{equation}\label{522}
\r_\delta\rightarrow\r \textrm{ in } C([0,T]; L^\gamma_{weak}(\O)),
\end{equation}
\begin{equation}\label{523}
\u_\delta\rightarrow\u  \textrm{ weakly in } L^2([0,T]; W^{1,2}_0(\O)),
\end{equation}
and thus
\begin{equation}\label{524}
\r_\delta\u_\delta\rightarrow\r\u \textrm{ weakly-* in } L^\infty([0,T]; L^{\f{2\gamma}{1+\gamma}}(\O)),
\end{equation}
where the limit functions $\r\geq 0$, $\u$ satisfy the continuity
equation \eqref{11} in $\mathcal{D}'(\O\times(0,T))$. Similarly,
since $\r_\delta$, $\u_\delta$, $\t_\delta$, $\H_\delta$ satisfy
the momentum equation \eqref{48}, we have
$$\r_\delta\u_\delta\rightarrow\r\u \textrm{ in } C([0,T]; L^{\f{2\gamma}{1+\gamma}}_{weak}(\O))$$
and
$$\r_\delta\u_\delta\otimes\u_\delta\rightarrow \r\u\otimes\u
\textrm{ weakly in } L^2([0,T]; L^{\f{6\gamma}{3+4\gamma}}(\O)).$$

From the hypothesis \eqref{23}, equation
\eqref{14}, and the estimates \eqref{51}, \eqref{513}, we can assume
$$p_\t(\r_\delta)\rightarrow \overline{p_\t(\r)} \textrm{ weakly in } L^\infty([0,T]; L^3(\O)),$$
$$\H_\delta\rightarrow\H \textrm{ weakly in } L^2([0,T]; W^{1,2}_0(\O))\cap C([0,T];L^2_{weak}(\O)),$$
with $\Dv\H=0$ in $\mathcal{D}'(\O\times(0,T))$. Hence, $\r$,
$\r\u$, $\H$ satisfy the initial data \eqref{116}. Due to the
estimate \eqref{514}, we can also assume
$$\t_\delta\rightarrow\t \textrm{ weakly in } L^2([0,T]; W^{1,2}(\O)),$$
with $\theta\geq 0$ in $\mathcal{D}'(\O\times(0,T))$, since
$$|\nabla\t_\delta|\leq (1+\t_\delta)^{\f{\alpha-1-\omega}{2}}|\nabla\t_\delta|, \textrm{ for } \omega\in(0,1).$$
Thus
\begin{equation}\label{525}
\t_\delta p_\t(\r_\delta)\rightarrow\t \overline{p_\t(\r) } \textrm{ weakly in } L^2([0,T]; L^{2}(\O)),
\end{equation}
\begin{equation}\label{526}
(\nabla\times\H_\delta)\times\H_\delta\rightarrow(\nabla\times\H)\times\H \textrm{ in } \mathcal{D}'(\O\times(0,T)).
\end{equation}
Similarly,
\begin{equation}\label{527}
\nabla\times(\u_\delta\times\H_\delta)\rightarrow\nabla\times(\u\times\H) \textrm{ in } \mathcal{D}'(\O\times(0,T)),
\end{equation}
\begin{equation}\label{528}
\nabla\times(\nu\nabla\times\H_\delta)\rightarrow\nabla\times(\nu\nabla\times\H) \textrm{ in } \mathcal{D}'(\O\times(0,T)).
\end{equation}
In view of \eqref{59} and the hypothesis \eqref{26}, we can assume
\begin{equation}\label{529}
Q_h(\t_\delta)\rightarrow \overline{Q_h(\t)} \textrm{ weakly in } L^2([0,T]; W^{1,2}(\O)),
\end{equation}
\begin{equation}\label{530}
M(\t_\delta)\rightarrow \overline{M(\t)} \textrm{ weakly in } L^2([0,T]; W^{1,2}(\O)),
\end{equation}
for any $M\in C^1[0,\infty)$ satisfying the growth restriction
$$|M'(\xi)|\leq c(1+\xi^{\f{\alpha}{2}-1}),$$
and, consequently,
\begin{equation}\label{531}
(\r_\delta+\delta) Q_h(\t_\delta)\rightarrow\r \overline{Q_h(\t)}  \textrm{ weakly in } L^2([0,T]; L^{\f{6\gamma}{\gamma+6}}(\O)),
\end{equation}
since $L^\gamma \hookrightarrow\hookrightarrow W^{-1,2}(\O)$, if
$\gamma>\f{3}{2}$.

At this stage we  need a variant of the celebrated
Aubin-Lions Lemma (cf. Lemma 6.3 in \cite{f2}):
\begin{Lemma}\label{l52}
Let $\{\t_n\}_{n=1}^\infty$ be a sequence of functions such that
$$\{\t_n\}_{n=1}^\infty \textrm{ is bounded in } L^2([0,T]; L^q(\O))\cap
L^\infty([0,T];L^1(\O)), \textrm{ with } q>\f{6}{5},$$
and assume that
$$\partial_t\t_n\geq\chi_n \textrm{ in } \mathcal{D}'(\O\times(0,T)),$$
where
$$\chi_n \textrm{ are bounded in } L^1([0,T]; W^{-m, r}(\O))$$
for certain $m\geq 1$, $r>1$.
Then $\{\t_n\}_{n=1}^\infty$ contains a subsequence such that
$$\t_n\rightarrow\t \textrm{ in } L^2([0,T]; W^{-1,2}(\O)).$$
\end{Lemma}

With this lemma in hand, we can show the following property:
\begin{Lemma}\label{l53}
Let $h=\f{1}{1+\t}$, then
$$(\r_\delta+\delta) Q_h(\t_\delta)\rightarrow\r \overline{Q_h(\t)}
\textrm{ in } L^2([0,T]; W^{-1,2}(\O)).$$
\end{Lemma}
\begin{proof}
Substituting $h=\f{1}{1+\t}$ into \eqref{411}, we get
\begin{equation*}
\begin{split}
\partial_t\left((\r_\delta+\delta) Q_h(\t_\delta)\right)\geq &-\Dv(\r_\delta Q_h(\t_\delta)
\u_\delta)+\Dv\left(\f{\kappa(\t_\delta)}{1+\t_\delta}\nabla\t_\delta\right)\\&-\delta
\f{\t_\delta^{\alpha+1}}{1+\t_\delta}-
\f{\t_\delta}{1+\t_\delta}p_\t(\r_\delta)\Dv\u_\delta,
\end{split}
\end{equation*}
in $\mathcal{D}'(\O\times(0,T))$.
Since
$$\f{\t_\delta}{1+\t_\delta}p_\t(\r_\delta)|\Dv\u_\delta|\leq p_\t(\r_\delta)|\Dv\u_\delta|,$$
we know that, in view of \eqref{23} and \eqref{51},
$$\f{\t_\delta}{1+\t_\delta}p_\t(\r_\delta)\Dv\u_\delta \textrm{ is bounded in }
L^2([0,T]; L^r(\O)), \textrm{ for some } r>1,$$
and, consequently,
$$\f{\t_\delta}{1+\t_\delta}p_\t(\r_\delta)\Dv\u_\delta \textrm{ is bounded in }
L^2([0,T]; W^{-k,r}(\O)), \textrm{ for all } k\geq 1.$$
Similarly, by \eqref{a7}
$$\delta\f{\t_\delta^{\alpha+1}}{1+\t_\delta} \textrm{ is bounded in } L^2([0,T];
W^{-k,r}(\O)), \textrm{ for all } k\geq 1,$$
and
$$\Dv(\r_\delta Q_h(\t_\delta)\u_\delta) \textrm{ is bounded in } L^1([0,T];
W^{-1,r}(\O)), \textrm{ for some } r> 1.$$

Next, by \eqref{25}, we have
$$\f{\kappa(\t_\delta)}{1+\t_\delta}|\nabla\t_\delta|\leq c(1+\t_\delta)^{\alpha-1}
|\nabla\t_\delta|\leq c\t_\delta^{\f{\alpha}{2}}|\nabla\t_\delta^{\f{\alpha}{2}}|,
\textrm{ if } \t_\delta\geq 1,$$
and
$$\f{\kappa(\t_\delta)}{1+\t_\delta}|\nabla\t_\delta|\leq c(1+\t_\delta)^{\alpha-1}
|\nabla\t_\delta|\leq c|\nabla\t_\delta|, \textrm{ if } \t_\delta\leq 1,$$
thus, by \eqref{59} and \eqref{a7}
$$\Dv\left(\f{\kappa(\t_\delta)}{1+\t_\delta}\nabla\t_\delta\right)\textrm{ is bounded in }
L^1([0,T]; W^{-1,r}(\O)), \textrm{ for some } r> 1.$$

Finally, since $Q_h(\t)\leq Q(\t)$, by \eqref{53}, we deduce that $$(\r_\delta+\delta)
Q_h(\t_\delta)\in L^\infty([0,T];L^1(\O)).$$
Hence, combining Lemma \ref{l52} and \eqref{531} together, we have
$$(\r_\delta+\delta) Q_h(\t_\delta)\rightarrow\r \overline{Q_h(\t)}  \textrm{ in }
L^2([0,T]; W^{-1,2}(\O)).$$
\end{proof}

Lemma \ref{l53} and \eqref{530} imply
\begin{equation}\label{532}
(\r_\delta+\delta) Q_h(\t_\delta)M(\t_\delta)\rightarrow\r \overline{Q_h(\t)}
\quad\overline{M(\t)} \textrm{ in } L^1(\O\times(0,T)),
\end{equation}
where $h(\t)=\f{1}{1+\t}$.
On the other hand, choosing $M(\t)=\t$, then $\t Q_h(\t)$ satisfies \eqref{530} since $\alpha>2$. Hence,
\begin{equation}\label{533}
(\r_\delta+\delta) Q_h(\t_\delta)\t_\delta\rightarrow \r
\overline{\t Q_h(\t)}\textrm{ weakly in } L^1(\O\times(0,T)).
\end{equation}
Properties \eqref{532} and \eqref{533} implies
$$\overline{\t Q_h(\t)}=\overline{Q_h(\t)}\t, \textrm{ a.e. on }\{\r>0\},$$
which yields
\begin{equation}\label{534}
\t_\delta\rightarrow\t \textrm{ in } L^{1}(\O\times(0,T)).
\end{equation}
Indeed, we know that $Q_h(\t)$ is strictly increasing and its
derivative has upper bound, therefore its inverse $Q_h^{-1}(\t)$
exists and has lower bound ${1/\overline{c_\upsilon}}$. Thus,
\begin{equation*}
\begin{split}
&\int_0^T\!\!\!\!\int_{\O}|Q_h(\t)-Q_h(\t_\delta)|^2\,dxdt\\
&\leq \overline{c_\upsilon}\int_0^T\!\!\!\!\int_{\O}(Q_h^{-1}(Q_h(\t))-Q_h^{-1}(Q_h(\t_\delta)))
(Q_h(\t)-Q_h(\t_\delta))\,dxdt\\
&= \overline{c_\upsilon}\int_0^T\!\!\!\!\int_{\O}(\t-\t_\delta)
(Q(\t)-Q(\t_\delta))\,dxdt\rightarrow 0, \textrm{ as }
\delta\rightarrow 0.
\end{split}
\end{equation*}
Therefore,
$$Q_h(\t_\delta)\rightarrow Q_h(\t), \textrm{ in } L^2(\O\times(0,T)),
\quad\textrm{ as } \delta\rightarrow 0,$$
and, hence,
$$Q_h(\t_\delta)\rightarrow Q_h(\t), \textrm{ a.e. in } \O\times(0,T),
\quad\textrm{ as } \delta\rightarrow 0.$$
Because $Q_h^{-1}(\t)$ is continuous, we deduce that
$$\t_\delta=Q_h^{-1}(Q_h(\t_\delta))\rightarrow \t=Q_h^{-1}(Q_h(\t)),
\textrm{ a.e. in } \O\times(0,T),\quad\textrm{ as } \delta\rightarrow 0,$$
which, combining Egorov's theorem, Theorem 2.10 in \cite{f2}, and the
weak convergence of $\{\t_\delta\}$ to $\t$ in $L^1(\O\times(0,T))$,
verifies \eqref{534}.

Finally, \eqref{534}, together with \eqref{523} and Lemma \ref{l41}, implies
\begin{equation}\label{535}
\Psi_\delta=\mu(\t_\delta)(\nabla\u_\delta+\nabla\u_\delta^T)+\lambda(\t_\delta)\Dv\u_\delta\mathbf{I}\rightarrow
\Psi=\mu(\t)(\nabla\u+\nabla\u^T)+\lambda(\t)\Dv\u\mathbf{I},
\end{equation}
in $\mathcal{D}'(\O\times(0,T))$, and,
\begin{equation}\label{536}
\int_0^T\!\!\!\!\int_{\O}h(\t_\delta)\Psi_\delta:\nabla\u_\delta\,dxdt\geq
\int_0^T\!\!\!\!\int_{\O}h(\t)\Psi:\nabla\u\,dxdt.
\end{equation}

\subsection{ Strong convergence of the density}

In this subsection, we will adopt the technique  in
\cite{f2} to show the strong convergence of the density,
specifically,
\begin{equation}\label{537}
\r_\delta\rightarrow\r \textrm{ in } L^1(\O\times(0,T)).
\end{equation}

First, due to \eqref{517}, Proposition 2.1 in \cite{f2} and \eqref{525}, we can assume that
\begin{equation}\label{538}
p(\r_\delta, \t_\delta)\rightarrow \overline{p(\r,\t)} \textrm{ weakly in } L^1(\O\times(0,T)),
\end{equation}
which, together with  \eqref{521}-\eqref{528}, implies that
\begin{equation}\label{539}
\rho_t +\Dv(\r\u)=0,
\end{equation}
\begin{equation}\label{540}
\partial _t(\r \u)+\Dv(\r \u\otimes \u) + \nabla \ov{p(\r,\t)}=(\na \times \H)\times \H+\Dv\Psi,
\end{equation}
\begin{equation}\label{541}
\H_t-\nabla\times(\u\times\H)=-\nabla\times(\nu\nabla\times\H),\quad \Dv\H=0,
\end{equation}
in $\mathcal{D}'(\O\times(0,T))$.

Similarly to Section 5.3 in \cite{hw1}, we use
$$\varphi_i(x,t)=\psi(t)\phi(x)\mathcal{A}_i[T_k(\r_\dl)], \, \psi\in
\mathcal{D}(0,T), \, \phi\in \mathcal{D}(\O), \, i=1,2,3,$$ as
test functions for the modified momentum balance equation
\eqref{48}, where $\mathcal{A}_i$ can be expressed by their
Fourier symbol as
$$\mathcal{A}_i[\cdot]
=\mathcal{F}^{-1}\left[\f{-\mathrm{i}\xi_i}{|\xi|^2}\mathcal{F}[\cdot]\right],
\, i=1,2,3,$$
and $T_k(\r)$ are cut-off functions,
$$ T_k(\r)=\min\{\r,k\}, \, k\geq 1.$$
A lengthy but straightforward computation shows that
\begin{equation}\label{542}
\begin{split}
&\int_0^T\!\!\!\!\int_{\O}\psi\phi\left((p(\r_\delta,\t_\delta)+\dl\r_\dl^\beta-\lambda(\t_\delta)\Dv
\u_\dl)T_k(\r_\dl)-2\mu(\t_\delta)\f{\partial\u^i_\delta}{\partial
x_j}\mathcal{R}_{i,j}[T_k(\r_\dl)]\right)\,dxdt\\
&=\int_0^T\!\!\!\!\int_{\O}\psi\partial_{x_i}\phi\left(\lambda(\t_\delta)\Dv
\u_\dl-p(\r_\delta,\t_\delta)-\dl\r_\dl^\beta\right)\mathcal{A}_i
[T_k(\r_\dl)]\,dxdt\\
&\quad +
\int_0^T\!\!\!\!\int_{\O}\psi\mu(\t_\delta)\left(\f{\partial\u^i_\delta}{\partial
x_j}+\f{\partial\u^j_\delta}{\partial x_i}\right)\partial_{x_j}\phi
\mathcal{A}_i[T_k(\r_\dl)]\,dxdt\\
&\quad -
\int_0^T\!\!\!\!\int_{\O}\phi\rho_\dl u_\dl^i\Big(\psi_t
\mathcal{A}_i[T_k(\r_\dl)]+\psi
\mathcal{A}_i[(T_k(\r_\dl)-T'_k(\r_\dl)\r_\dl)\Dv
\u_\dl]\Big)dxdt\\
&\quad - \int_0^T\!\!\!\!\int_{\O}\psi\r_\dl
u^i_\dl u_\dl^j\partial_{x_j}\phi
\mathcal{A}_i[T_k(\r_\dl)]\,dxdt\\
&\quad +\int_0^T\!\!\!\!\int_{\O}\psi
u_\dl^i\left(T_k(\r_\dl)\mathcal{R}_{i,j}[\r_\dl u^j_\dl]-\phi \r_\dl
u^j_\dl \mathcal{R}_{i,j}[T_k(\r_\dl)]\right)dxdt\\
&\quad -\int_0^T\!\!\!\!\int_{\O}\psi\phi(\na \times \H_\dl)\times
\H_\dl\cdot\mathcal{A}[T_k(\r_\dl)]\,dxdt,
\end{split}
\end{equation}
where the operators
$\mathcal{R}_{i,j}=\partial_{x_j}\mathcal{A}_i[v]$ and the summation
convention is used to simplify notations.
On the other hand, following  the arguments of Section 5.3 in \cite{hw1}, one has
\begin{equation}\label{543}
\begin{split}
&\int_0^T\!\!\!\!\int_{\O}\psi\phi\left(\left(\ov{p(\r,\t)}-\overline{\lambda(\t)\Dv
\u}\right)\ov{T_k(\r)}-2\ov{\mu(\t)\f{\partial\u^i}{\partial
x_j}}\mathcal{R}_{i,j}[\ov{T_k(\r)}]\right)dxdt\\
&=
\int_0^T\!\!\!\!\int_{\O}\psi\partial_{x_i}\phi\left(\ov{\lambda(\t)\Dv
\u}-\ov{p(\r,\t)}\right)\mathcal{A}_i [\ov{T_k(\r)}]\,dxdt\\
&\quad +
\int_0^T\!\!\!\!\int_{\O}\psi\ov{\mu(\t)\left(\f{\partial\u^i}{\partial
x_j}+\f{\partial\u^j}{\partial x_i}\right)}\partial_{x_j}\phi
\mathcal{A}_i[\ov{T_k(\r)}]\,dxdt\\
&\quad -
\int_0^T\!\!\!\!\int_{\O}\phi\r u^i\left(\psi_t
\mathcal{A}_i[\ov{T_k(\r)}]+\psi
\mathcal{A}_i[\ov{(T_k(\r)-T'_k(\r)\r)\Dv \u}]\right)dxdt\\
&\quad -
\int_0^T\!\!\!\!\int_{\O}\psi\rho u^i u^j\partial_{x_j}\phi
\mathcal{A}_i[\ov{T_k(\r)}]\,dxdt\\
&\quad +
\int_0^T\!\!\!\!\int_{\O}\psi
u^i\left(\ov{T_k(\r)}\mathcal{R}_{i,j}[\phi\r u^j]-\phi \r u^j
\mathcal{R}_{i,j}[\ov{T_k(\r)}]\right)dxdt\\
&\quad -
\int_0^T\!\!\!\!\int_{\O}\psi\phi(\na \times \H)\times
\H\cdot\mathcal{A}[\ov{T_k(\r)}]\,dxdt.
\end{split}
\end{equation}
Now, following the argument in Section 5.3 in \cite{hw1}, the
Div-Curl Lemma can be used in order to show that the right-hand side
of \eqref{542} converges to that of \eqref{543}, that is
\begin{equation}\label{544}
\begin{split}
&\lim_{\delta\rightarrow
0}\int_0^T\!\!\!\!\int_{\O}\psi\phi\left((p(\r_\delta,\t_\delta)-\lambda(\t_\delta)\Dv
\u_\dl)T_k(\r_\dl)-2\mu(\t_\delta)\f{\partial\u^i_\delta}{\partial
x_j}\mathcal{R}_{i,j}[T_k(\r_\dl)]\right)dxdt\\&=
\int_0^T\!\!\!\!\int_{\O}\psi\phi\left((\ov{p(\r,\t)}-\overline{\lambda(\t)\Dv
\u})\ov{T_k(\r)}-2\ov{\mu(\t)\f{\partial\u^i}{\partial
x_j}}\mathcal{R}_{i,j}[\ov{T_k(\r)}]\right)dxdt.
\end{split}
\end{equation}
Noting that
\begin{equation*}
\begin{split}
&\int_0^T\!\!\!\!\int_{\O}\varphi\mu(\t_\delta)\f{\partial\u^i_\delta}{\partial
x_j}\mathcal{R}_{i,j}[T_k(\r_\dl)]\,dxdt\\&=
\int_0^T\!\!\!\!\int_{\O}\left(\mathcal{R}_{i,j}\left[\varphi\mu(\t_\delta)
\f{\partial\u^i_\delta}{\partial x_j}\right]-\varphi\mu(\t_\delta)\mathcal{R}_{i,j}
\left[\f{\partial\u^i_\delta}{\partial x_j}\right]\right) T_k(\r_\dl)\,dxdt\\
&\quad +\int_0^T\!\!\!\!\int_{\O}\varphi\mu(\t_\delta)\Dv\u_\delta
T_k(\r_\delta)\,dxdt,
\end{split}
\end{equation*}
for any $\varphi\in\mathcal{D}(\O\times(0,T))$, we have,
using  also \eqref{544},
\begin{equation}\label{545}
\begin{split}
&\lim_{\delta\rightarrow
0}\int_0^T\!\!\!\!\int_{\O}\varphi\big((p(\r_\delta,\t_\delta)-
(\lambda(\t_\delta)+2\mu(\t_\delta))\Dv
\u_\dl)T_k(\r_\dl)\big)dxdt\\&=
\int_0^T\!\!\!\!\int_{\O}\varphi\left(\left(\ov{p(\r,\t)}-\overline{(\lambda(\t)+2\mu(\t))\Dv
\u}\right)\ov{T_k(\r)}\right)dxdt,
\end{split}
\end{equation}
since Lemma 4.2 in \cite{f1} and the strong convergence of the
temperature give
\begin{equation*}
\begin{split}
\int_0^T\!\!\!\!\int_{\O}&\left(\mathcal{R}_{i,j}\left[\varphi\mu(\t_\delta)
\f{\partial\u^i_\delta}{\partial
x_j}\right]-\varphi\mu(\t_\delta)\mathcal{R}_{i,j}
\left[\f{\partial\u^i_\delta}{\partial x_j}\right]\right)
T_k(\r_\dl)\,dxdt\\&\rightarrow\int_0^T\!\!\!\!\int_{\O}\left(\mathcal{R}_{i,j}\left[\varphi\mu(\t)
\f{\partial\u^i}{\partial
x_j}\right]-\varphi\mu(\t)\mathcal{R}_{i,j}
\left[\f{\partial\u^i}{\partial x_j}\right]\right) T_k(\r)\,dxdt.
\end{split}
\end{equation*}
As in Section 5.4 of \cite{hw1}, we can conclude from \eqref{545}
that there exists a constant $c$ independent of $k$ such that
$$\limsup_{\dl\rightarrow 0+}\|
T_k(\r_\dl)-T_k(\r)\|_{L^{\gamma+1}((0,T)\times \O)}\leq c.$$ This
implies, in particular, that the limit functions $\r$, $\u$ satisfy
the continuity equation \eqref{11} in the sense of renormalized
solutions (cf. Lemma 5.4 in \cite{hw1}).

Finally, following the argument as in Section 5.6 in \cite{hw1}, \eqref{537} is verified.

\subsection{ Thermal energy equation}

In order to complete the proof of Theorem \ref{mt}, we have to show
that $\r$, $\u$, $\t$ and $\H$ satisfy the thermal energy equation
\eqref{115} in the sense of Definition \ref{d1}.

In view of \eqref{a7} and \eqref{534}, we have, as $\delta\to 0$,
$$\t_\dl\rightarrow\t \textrm{ in } L^p(\O\times(0,T)), \textrm{ for all } 1\leq p<1+\alpha.$$
Hence, by the Lebesgue's dominated convergence theorem and the hypothesis \eqref{25}, we know, as $\delta\to 0$,
$$K_h(\t_\dl)\rightarrow K_h(\t) \textrm{ in } L^1(\O\times(0,T)). $$
By \eqref{522}, \eqref{523} and \eqref{537}, we have, as $\delta\to 0$,
$$\r_\dl Q_h(\t_\dl)\u_\dl\rightarrow\r Q_h(\t)\u,$$
$$\r_\dl Q_h(\t_\dl)\rightarrow \r Q_h(\t),$$
$$h(\t_\dl)\t_\dl p_\t(\r_\dl)\Dv\u_\dl\rightarrow h(\t)\t p_\t(\r)\Dv\u,$$
in $\mathcal{D}'(\O\times(0,T))$.

Due to the strong convergence of the temperature \eqref{534} and
\eqref{536}-\eqref{537}, we can pass the limit as
$\delta\rightarrow 0$ in \eqref{411} to obtain
\begin{equation}\label{546}
\begin{split}
&\int_0^T\!\!\!\!\int_{\O}\left(\r Q_h(\t)\partial_t\varphi+\r Q_h(\t)\u\cdot\nabla\varphi+K_h(\t)\Delta\varphi\right) dxdt\\
&\leq-\int_0^T\!\!\!\!\int_{\O}(h(\t)(\Psi:\nabla\u+\nu|\nabla\times\H|^2)\varphi\,dxdt\\
&\quad +\int_0^T\!\!\!\!\int_{\O}h(\t)\t
p_{\t}(\r)\Dv\u\varphi\,dxdt-\int_{\O}\r_{0}Q_h(\t_{0})\varphi(0)\,dx,
\end{split}
\end{equation}
since
$$\delta\left|\int_0^T\!\!\!\!\int_{\O}\t_\delta^{1+\alpha}h(\t_\delta)dxdt\right|\leq \delta \left|\int_{\{\t_\delta\leq
M\}}\t_\delta^{1+\alpha}dxdt\right|+h(M)\delta\int_0^T\!\!\!\!
\int_{\O}\t_\delta^{1+\alpha}dxdt,$$ which tends to zero
as $\delta\rightarrow 0$, because the first term on the right-hand
side tends to zero for fixed M as $\delta\rightarrow 0$ while the
second term can be made arbitrarily small by taking M large enough
in view of \eqref{44} and \eqref{55}.

Next, taking
$$h(\t)=\f{1}{(1+\t)^\omega}, \quad 0<\omega<1,$$
in \eqref{546}, letting $\omega\rightarrow 0$, and using
Lebesgue's dominated convergence theorem, we get
\begin{equation}\label{547}
\begin{split}
&\int_0^T\!\!\!\!\int_{\O}\left(\r Q(\t)\partial_t\varphi+\r Q(\t)\u\cdot\nabla\varphi+K(\t)
\Delta\varphi\right)dxdt\\
&\leq \int_0^T\!\!\!\!\int_{\O}(\t
p_{\t}(\r)\Dv\u-\nu|\nabla\times\H|^2-\Psi:\nabla\u)\varphi\,dxdt,
\end{split}
\end{equation}
for any $\varphi\in \mathcal{D}(\O\times(0,T))$ and $\varphi\geq 0$,
since $\r Q_h(\t)\leq\r Q(\t) $ belongs to $L^1(\O\times(0,T))$ by
\eqref{51}, \eqref{26}, and \eqref{515}.

Finally, dividing \eqref{21} by $1+\t$ and using equation
\eqref{11}, we get
\begin{equation*}
\begin{split}
&\partial_t(\r f(\t))+\Dv(\r f(\t)\u)+\Dv\left(\f{q}{1+\t}\right) \\
&\geq\f{1}{1+\t}(\nu|\nabla\times\H|^2
+\Psi:\nabla\u)-\f{q\cdot\nabla\t}{(1+\t)^2}
-\f{\t}{1+\t}p_\t(\r)\Dv\u,
\end{split}
\end{equation*}
 in the sense of distributions, where
$$f(\t)=\int_0^\t\f{c_\nu(\xi)}{1+\xi}\,d{\xi}.$$
Integrating the above inequality over $\O\times(0,T)$, we deduce
\begin{equation}\label{548}
\begin{split}
&\int_0^T\!\!\!\!\int_\O \left(\f{1}{1+\t}(\nu|\nabla\times\H|^2+\Psi:\nabla\u)+\f{k(\t)|\nabla\t|^2}{(1+\t)^2}
\right)dxdt\\
&\le 2\sup_{0\le t\le T}\int_\O\r f(\t)\,dx+
\int_0^T\!\!\!\!\int_\O\f{\t}{1+\t} p_\t(\r)|\Dv\u|\,dxdt.
\end{split}
\end{equation}
By H\"{o}lder's inequality, the estimates \eqref{51}, \eqref{513},
and the hypothesis \eqref{23}, one has
$$\int_0^T\!\!\!\!\int_\O\f{\t}{1+\t}
p_\t(\r)|\Dv\u|\,dx\,dt\le\int_0^T\!\!\!\!\int_\O
p_\t(\r)|\Dv\u|\,dx\,dt\le c.$$
Similarly, by H\"{o}lder's
inequality, the assumption \eqref{26}, and the estimates \eqref{51},
\eqref{a7}, we have
$$\int_\O\r f(\t)\,dx\le c\int_\O\r \t\,dx\le c.$$
Thus, \eqref{548} and the assumption \eqref{25} imply that
$$\ln(1+\t)\in  L^2([0,T]; W^{1,2}(\O)), \quad
\t^{\f{\alpha}{2}}\in L^2([0,T]; W^{1,2}(\O)).$$
This completes our proof of Theorem \ref{mt}.

\bigskip\bigskip

\section*{Acknowledgments}

Xianpeng Hu's research was supported in part by the National Science Foundation grant  DMS-0604362.
Dehua Wang's research was supported in part by the National Science
Foundation grants DMS-0244487, DMS-0604362, and the Office of Naval
Research grant N00014-01-1-0446.

\end{document}